\documentclass{article}
\usepackage{graphics}
\usepackage{amssymb}
\usepackage{amsmath}
\usepackage{color,soul}
\usepackage{enumerate}
\usepackage{mathrsfs}
\usepackage{tikz}
\usetikzlibrary{matrix}
\usetikzlibrary{arrows}
\usetikzlibrary{shapes}
\usepackage{enumitem}

\usepackage[affil-it]{authblk}

\usepackage{amsthm}
\newtheorem{thm}{Theorem}
\numberwithin{thm}{section}
\newtheorem{lem}[thm]{Lemma}
\newtheorem{defn}[thm]{Definition}
\newtheorem{cor}[thm]{Corollary}
\newtheorem{prop}[thm]{Proposition}
\newtheorem{remk}[thm]{Remark}

\newtheorem{claim}[thm]{Claim}
\let\oldremark\remk
\renewcommand{\remk}{\oldremark\normalfont}
\let\olddefn\defn
\renewcommand{\defn}{\olddefn\normalfont}

\def \conf {{\mathrm{Conf}}}
\def \harm {{\mathcal{H}}}
\newcommand{\cev}[1]{\overleftarrow{#1}}

\begin{document}

\title{Thermodynamic Formalism for Transient Potential Functions}
\author{Ofer Shwartz
\thanks{ofer.shwartz@weizmann.ac.il}
}                     

\affil{Faculty of Mathematics and Computer Science,\\ The Weizmann Institute of Science
}

\maketitle
\begin{abstract}
We study the thermodynamic formalism of locally compact Markov shifts with transient potential functions. In particular, we show that the Ruelle operator admits positive continuous eigenfunctions and positive Radon eigenmeasures in forms of Martin kernels. These eigenmeasures can be characterized in terms of the direction of escape to infinity of their orbits, when viewed inside a suitable Martin-like compactification of the underlying shift space. We relate these results to first-order phase transitions in one-dimensional lattice gas models with infinite set of states. This work complements earlier works by Sarig \cite{sarig_1999,sarig_2001} who focused on the recurrent scenario.

\end{abstract}
\section{Introduction}
The main tool in the study of Thermodynamic Formalism for topological Markov shifts (or shortly TMS, see definition in  Section \ref{section:preliminaries}) is the Ruelle operator
$$ (L_\phi f)(x) = \sum_{y:Ty=x}e^{\phi(y)}f(y)$$
and in particular its eigenfunctions and eigenmeasures, see for example  
\cite{bowen_1975,sarig_1999,mauldin_2001,stratmann_2007,ruelle_2004,lanford_1969,aaronson_1993,baladi_2000}. 

For a topologically transitive one-sided TMS $(X^{+},T)$ with finite set of states $S$ and a H\"older continuous function $\phi:X^+\rightarrow\mathbb{R}$, Ruelle \cite{ruelle_2004} and Bowen \cite{bowen_1975}   showed that there is a positive continuous eigenfunction $L_{\phi }h = \lambda h$ and a positive  eigenmeasure $L_\phi^*\nu=\lambda \nu$ with $\nu(h)=1 $ and $\log \lambda$ is the pressure of $\phi$. The eigenvectors $
\nu, h$ are unique up to scaling and the measure $\mu = h \nu$ is the unique equilibrium state which maximizes $h_\mu(T) + \mu(\phi)$.
The eigenmeasures of the Ruelle operator are also called \textit{conformal measures} and  their Jacobian $\frac{d\nu}{d\nu \circ T}$ is $\lambda^{-1}\exp \phi$. 

For a topologically transitive TMS with countable number of states, $|S|=\infty$, Sarig \cite{sarig_2001} showed that the behaviour of a H\"older continuous potential function $\phi$ with finite Gurevich pressure can be characterized either as  positive recurrent, null recurrent or transient. If $\phi$ is positive recurrent, then the situation is similar to the finite case: the eigenfunction, eigenmeasure and the equilibrium state exist and are unique if $X^+$ is topologically mixing \cite{buzzi_2003}. If $\phi$ is null recurrent, $h$ and $\nu$ still exist and unique but now  $hd\nu$ is an infinite conservative measure, which makes the discussion on entropy and equilibrium states more subtle. 

As for  transient $\phi$, for a locally compact $X^+$, Cyr \cite{cyr_2010} showed the existence of a totally-dissipative eigenmeasure, as  a weak$^*$-converging sub-sequence of
$$\mu_k(A) := \frac{\sum_{n=0}^\infty L_\phi^n 1_A(T^kx)}{\sum_{n=0}^\infty L_\phi^n 1_{[o]}(T^kx)} $$
where $T^kx\xrightarrow[k\rightarrow\infty]{} \infty$ (escapes every compact set) and $o\in S$ is  arbitrary. 
Later on, motivated by the study of KMS states,  Thomsen had  analyzed the class of conformal measures of Markovian potentials on countable-states Markov shifts, and produced examples where the dependence of this class on the potential exhibits a phase transition, see  \cite{thomsen_2016} and also \cite{thomsen_2017}.
  Stadlbauer  used the Patterson-Sullivan approach to construct eigenmeasures and eigenfunctions for a class of topological Markov shifts obtained from countable group extensions of topological Markov shifts with the big images and pre-images property, see \cite{stadlbauer_2017}.  For more on the Thermodynamic Formalism of a transient potential function, see \cite{iommi_2013}.

The purpose of this paper is to complete the analysis on the eigenmeasures and the eigenfunctions in the transient case. For a locally compact topologically transitive Markov shift equipped with a $\lambda$-transient potential function (see definitions in Section \ref{section:preliminaries}), we show the following:
\begin{enumerate}
\item 
\underline{Existence of eigenvectors:} There exist a positive Radon measure $\mu$ and a positive continuous function $h$ s.t. $L_\phi^* \mu = \lambda \mu$ and $L_\phi h = \lambda h$. The existence of $\mu$ was shown before by Cyr \cite{cyr_2010}, but the existence of $h$ is, as far as we know, new.
Examples show that $\mu$ or $h$ need not be unique. 
\item
\underline{Representation of eigenvectors:} We extend Martin's representation theorem \cite{martin_1941} to the context of Ruelle operator. Specifically, we construct a compactification $\widehat{X^+}$ of $X^+$ with boundary $\mathcal{M}=\widehat{X^+} \setminus X^+ $ and construct a kernel $K(f,\omega|\lambda)$ ($f\in C_c^+(X^+), \omega \in \widehat{X^+}$) s.t. every Radon measure $\mu$ with $L_\phi^*\mu = \lambda \mu$ has the form $$\mu(\cdot) = \int_{\mathcal{M}}K(\cdot, \omega|\lambda)d\nu(\omega)$$ for some finite measure $\nu$ on $\mathcal{M}$. We provide a similar representation for positive eigenfunctions using a compactification of the negative one-sided Markov shift $X^-$ (see definitions in Section \ref{section:harmonic_functions}). \item
\underline{Direction of escape to $\infty$:} We show that every extremal $\lambda$-eigenmeasure can be represented by $\nu \propto \delta_\omega$, for some $\omega\in \mathcal{M}$. For this $\omega$, $$T^nx\xrightarrow[n\rightarrow\infty]{}\omega\quad  \mu\text{-a.e.}$$
so the extremal eigenmeasures are characterized by the almost-sure direction of escape to infinity of their orbits. 
\item
\underline{Duality:} We show that positive $\lambda$-eigenfunctions with uniformly continuous logarithm can be canonically identified with $\lambda$-eigenmeasures for a ``reversed" Ruelle operator on the  negative one-sided Markov shift $X^-$. This duality is valid for the recurrent case as well.
\item
\underline{Representation of dominated eigenfunctions}: We show that any $\lambda$-eigenfunction $f$ which is dominated by a positive $\lambda$-eigenfunction $h$ is fully
characterized by an almost-everywhere bounded function $\varphi$ on a suitable Poisson boundary and that $\varphi$ describes the limiting behaviour of $\frac{f}{h}$. 
\item
\underline{First-order phase transitions}:
We apply the main results of this work to the study of Gibbs states and first-order phase transitions. In particular, we show that a phase transition occurs when the Martin boundary has more than a single point and provide an analogues interpretation of a thermodynamic limit in the transient case.  
\end{enumerate}

As in \cite{cyr_2010,thomsen_2016}, our  approach is motivated by the theory of the Martin boundary for random walks. Recall that for a transient random walk on a  locally finite graph, one can represent every positive harmonic function in terms of  Martin kernels and show that the walk almost surely converges  to a boundary point, see Section \ref{section:random_walk}. We emphasize that unlike in the probabilistic settings, where a compactification of the set of states $S$ is considered, our proposed compactification is of the space of paths $X^+$. For an alternative approach of a compactification of Markov shifts, see \cite{fiebig_1995,fiebig_2013}.

\subsection{Topological Markov shifts, Ruelle operator and transience}
\label{section:preliminaries}
Let $S$ be an infinite countable set of states and let $\mathbb{A}\in \{0,1\}^{S\times S}$ be a transition matrix over $S$. For a subset  $A \subseteq \mathbb{Z}$ and a vector $x\in S^{A}$, we denote by $(x)_i$ the $i$-th coordinate of $x$. 

The \textit{(positive) one-sided topological Markov shift} (TMS) is the space
$$X^{+} =\{x\in S^{\mathbb{N}\cup \{0\}} : \mathbb{A}_{(x)_i, (x)_{i+1} }= 1, \forall i\geq 0\} $$
with the transformation $T:X^{+}\rightarrow X^{+}$, $(T(x))_i = (x)_{i+1}$ and the  metric
\begin{equation}
\label{eq:def_d}
d(x,y) = 2^{-\inf\{i\geq 0:(x)_i\neq (y)_i\}}.
\end{equation}

If $\sum_{b}\mathbb{A}_{a,b}<\infty$ for every $a\in S$, then the space $(X^{+},d)$ is locally compact  and all \textit{cylinder} sets
\begin{equation}
\label{eq:def_cylinders}
[a_0,\dots, a_m] := \{x\in X^{+} : (x)_i = a_i, 0\leq i \leq m\}
\end{equation}
are compact. A word $(a_0, \dots, a_m)\in S^{m+1}$ is called \textit{admissible} if $[a_0,\dots, a_m]\neq \varnothing$. 
We say that a TMS $X^{+}$ is \textit{topologically transitive}, or simply \textit{transitive}, if for every $a,b\in S$ there exists $n$ s.t. $T^{-n}[a] \cap [b] \neq \varnothing$. 
For every state $a\in S$ on a transitive TMS $X^{+}$, we arbitrarily pick $x_a\in T^{}[a]$. 

For two states $a,b\in S$ let 
$d_{graph}(a,b)$ be the directed-graph distance from  $a$ to $b$,
$$ d_{graph}(a,b) := \inf\{n\geq 0:[a] \cap T^{-n}[b]\}.$$
Observe that $d_{graph}$ is not necessarily a metric.
 
For  numbers $r_1,r_2,c\in \mathbb{R}^+$, we write $r_1 =e^{ \pm c} r_2$ if
$ e^{-c}r_2 \leq r_1 \leq e^{c} r_2$. We write $o_n(1)$ for a small quantity, converging to zero as $n$ tends to infinity.

 We denote by $C_c(X^{+})$  the space of all continuous functions from $X^+$ to $\mathbb{R}$ with compact support, by $C^+(X^{+})$  the space of all non-negative continuous functions and by $C^+_c(X^{+})=C^+(X^{+}) \cap C_c(X^{+})$ the space of all non-negative continuous functions with compact support.  

The $m$-th \textit{variation} of a  function $\phi:X^{+}\rightarrow\mathbb{R} $ is 
\begin{equation}
\label{eq:def_var} Var_m(\phi) = \inf\{|\phi(x) - \phi(y)| : x,y\in X^{+}, (x)_i=(y)_i, 0\leq i< m-1\}.
\end{equation}
Henceforth we always assume that  $\phi$  has \textit{summable variations}:   $\sum_{m\geq 2} Var_m(\phi) <\infty$. Notice that this is satisfied by all Markovian potential functions, $\phi(x) = \phi((x)_0, (x)_1)$, and by all H\"older-continuous functions as well.  

\begin{defn}
 The \textit{Ruelle operator} $L_\phi$ evaluated at a function $f\in C(X^{+})$ and a point $x\in X^+$ is
$$(L_\phi f)(x) = \sum_{y:Ty=x}e^{\phi(y)}f(y). $$
\end{defn}
Then, 
$$(L_\phi^nf)(x)=\sum_{y:T^{n}y=x}e^{\phi_{n}(y)}f(y) $$ 
where $\phi_n(x) = \sum_{i=0}^{n-1}\phi(T^ix)$. Observe that when $X^+$ is locally compact,  $(L_\phi^n f)(x)<\infty $ for all $n\geq 0,f\in C_c(X^{+})$ and $x\in X^+$.

The \textit{Gurevich pressure} of $\phi$ is defined to be the limit
$$P_G(\phi) = \limsup_{n\rightarrow\infty}\frac{1}{n}\log \sum_{T^n x=x}e^{\phi_n(x)}1_{[a]}(x) $$
for some $a\in S$ and $x\in X^{+}$. It can be shown that if $(X^{+},T)$ is topologically mixing then the limit exists and independent of  the choice
of $a$, see \cite{sarig_1999}. 
\begin{defn}The \textit{Green's function} for a function $f\in C^+(X^+)$ and  $\lambda>0$ evaluated at a point $x\in X^+$ is (the possibly divergent) sum
$$ G(f,x|\lambda) = \sum_{n=0}^\infty \lambda^{-n}\left(L_\phi^nf\right)(x).$$
\end{defn}

We now introduce the notion of $\lambda$-recurrence and $\lambda$-transience, for an arbitrary $\lambda \in (0,\infty)$.
\begin{defn}
A potential function $\phi$ with summable variations on a transitive one-sided TMS $(X^+,T)$ is called \textit{$\lambda$-recurrent} if for some (or every)  $f\in C_c^{+}(X)\setminus\{0\}$ and $x\in X^{+}$, $G(f,x|\lambda)=\infty$ and is called \textit{$\lambda$-transient} if for some (or every) $f\in C_c^{+}(X^{+})\setminus\{0\}$ and $x\in X^{+}$, $G(f,x|\lambda)<\infty$. 
\end{defn}
Notice that  $\phi$ is $\lambda$-recurrent for all $\lambda < \exp P_G(\phi)$  and $\lambda$-transient for all $\lambda > \exp P_G(\phi)$. 

 In this paper we adapt the notion of $\lambda$-recurrence and $\lambda$-transience as it appears in  common probability literature. In \cite{sarig_1999}, the term recurrent or transient is actually interpreted as $\exp (P_G(\phi))$-recurrent or $\exp(P_G(\phi))$-transient, with finite $P_G(\phi)$. To obtain more general results, we considered arbitrary value of $\lambda$, rather than the specific but important value  $\lambda = \exp(P_G(\phi))$. If $\phi$ is $\lambda$-transient  with $\lambda=1$, we simply say that $\phi$ is \textit{transient}. Then, we write $G(f,x) = G(f,x|1)$.

Recall that a measure $\mu$ is called a \textit{\textit{Radon} measure} if it is a Borel measure which is finite on compacts.
For two positive, possibly infinite, Radon measures $\mu$ and $\nu$ on $X^+$, we write $\mu\leq \nu$  if $\mu(K)\leq \nu(K)$, for every compact set $K\subseteq X^+$. We say that a measure $\mu$ is \textit{non-singular} if $\mu \circ T^{-1} \sim \mu$ i.e. for every measurable set $E$, $\mu(E)=0 \Leftrightarrow \mu(T^{-1}E)=0$. Recall that a sequence of Borel measures $\mu_n$ converges to a measure $\mu$ in the \textit{weak$^*$ topology} if for every $f\in C_c(X^+)$, $\mu_n(f)\xrightarrow[n\rightarrow\infty]{} \mu(f)$. 

For two Radon measures $\mu_1$ and $\mu_2$, we write $\mu_1 \propto \mu_2$ if  the two measures are proportional, i.e. there exists $c\in \mathbb{R}$ s.t. $\mu_1 = c \mu_2$. We use the same notation for functions as well: $f_1 \propto  f_2$ means that there is $c\in \mathbb{R}$ s.t. $f_1 = c f_2$. 
\begin{defn}A positive Radon measure $\mu$ on $X^+$ is called  \textit{$\lambda$-conformal} (or simply conformal) if $L_\phi^*\mu=\lambda\mu$, i.e. $(L_\phi^*\mu)(1_{K}) = \lambda\mu(1_{K})$ for every compact set $K\subseteq X^{+}$. Similarly,  $\mu$ is called  \textit{$\lambda$-excessive}  if $L_\phi^* \mu \leq \lambda\mu$.  We denote by $\conf(\lambda)$ the space of all positive Radon $\lambda$-conformal measures. If $X^+$ is transitive and $\mu \neq 0$ is conformal then for every state $a\in S$, $\mu([a])>0$; see Lemma \ref{lemma:positive_on_cylinders} in the appendix. 
\end{defn}

\subsection{Standing assumptions}
In this work, we make the following assumptions:
\begin{enumerate}[label={(A\arabic*)}]
\item
$(X^{+},T)$ is topologically transitive with $|S|=\infty$.
\item
\label{assumption:locally_compact}
For every $a\in S$, $\sum_{b}\mathbb{A}_{a,b}<\infty$ (equivalently, $X^+$ is locally compact). \item
$\phi$ has summable variations,  $\sum_{k=2}^\infty Var_k(\phi) < \infty.$
\item 
\label{assumption:finite_pressure}
There is $\lambda >0$ s.t. $\phi$ is $\lambda$-transient, i.e. $\sum \lambda^{-n} L_\phi ^n f(x) <\infty$ for every $f\in C_c^+(X^+)$ and $x\in X^+$ (equivalently for some $0\neq f \in C_c^+(X)$ and some $x\in X^+$). 
\end{enumerate} 
Assumption \ref{assumption:locally_compact} is crucial.
See \cite{cyr_2010} for an example of a non locally-compact Markov shift and a potential function with no conformal measures. Though assumption   \ref{assumption:locally_compact} excludes several interesting models, it is satisfied by  the symbolic models for non-uniformly hyperbolic diffeomorphisms in \cite{sarig_2013,ovadia_2016}.
By definition, assumption \ref{assumption:finite_pressure} is equivalent to the assumption that $P_G(\phi)<\infty$.

Observe that for all $\lambda>0$, 
$$L_{(\phi-\log \lambda)}^* \mu = \mu \text{ iff } L_\phi^* \mu = \lambda \mu.$$
Hence in order to study the $\lambda$-conformal measures of $\phi$ one can study the $1$-conformal measures of $\phi - \log\lambda $. We use this reduction frequently. 

In Section \ref{section:harmonic_functions} we shall consider two-sided and negative one-sided topological Markov shifts. The following assumption is essential to ensure that these topological Markov shifts are locally compact;
\begin{enumerate}[label={(A\arabic*)}]
\setcounter{enumi}{4}
 \item 
 \label{assumption:locally_compact_two_sides}
For every $b\in S$, $\sum_{a}\mathbb{A}_{a,b}<\infty$. 
\end{enumerate}

\section{Martin boundary and the existence of eigenmeasures}
\label{section:construction}
In this section we construct a special compactification of $X^+$ and use it to show the following result of Cyr \cite{cyr_2010}:
\begin{thm}
\label{thm:conf_existence}
Under assumptions (A1)-(A4), if $\sum_{n\geq 0}\lambda^{-n}(L_\phi^n f)(x)<\infty$ for some $\lambda>0, f\in C_c^+(X^+)$ and $x\in X^+$, then there exists a positive Radon measure $\mu$ s.t. $L_\phi^* \mu = \lambda \mu$,
\end{thm}
In the following two sections we will use this compactification to describe all conformal measures $\mu$ and to describe the asymptotic behaviour of $T^nx$ as $n\rightarrow\infty$ for $\mu$-typical $x\in X^+$. 

The construction is motivated by the well-known Martin compactification of transient Markov chains, but differs from it by at least one important aspect: we compactify the space of infinite paths $X^+$ and not the set of states $S$. See  \cite{sawyer_1997,woess_2000} for a detailed exposition of the probabilistic Martin boundary.

Recall the definition of Green's function; $G(f,x|\lambda) = \sum_{n\geq 0}\lambda^{-n}L_\phi^n f(x)$.

\begin{defn} Let $o\in S$ be an arbitrary ``origin state''. The \textit{Martin kernel} $K:C_c(X^{+})\times X^{+} \rightarrow\mathbb{R}^+$ is 
$$  K(f,x|\lambda) := \frac{  G(f,x|\lambda)}{  G(1_{[o]},x|\lambda)}. $$
\end{defn}
Later on we show that the choice of the origin state is not crucial (see Corollary \ref{cor:ind_of_ref_fucntion}). We write $  K(f,x)=  K(f,x|1)$.

When $X^+$ is locally compact and  $\phi$ has summable variations, the Green's function and the Martin kernel\ are continuous, see Lemma \ref{lemma:G_continuous} in the appendix.

To construct the compactification, we  introduce a new metric $\rho$ on $X^{+}$, which coincides with  the convergence according to the original metric $d$ and for which the Martin kernels $   K(f,x|\lambda)$ are $\rho$-continuous  for every $f\in C_c(X^{+})$.  The Martin boundary  is then the set of all new  points in the completion of $X^{+}$ w.r.t. $\rho$. 

We start by showing that for every fixed $f\in C_c(X^{+})$,  $  K(f,x|\lambda)$ is  bounded in $x$. 
\begin{lem}
\label{lemma:kernel_boundness}
Let $(X^{+},T)$ be a transitive locally compact one-sided TMS and let $\phi$ be a $\lambda$-transient potential function with summable variations. Then, for every $ a,b\in S$,  there exist $0<c_{a,b} \leq C_{a,b}$ s.t.
\begin{equation}
\label{eq:green_func_bound} c_{a,b}   G(1_{[b]},x|\lambda)\leq   G(1_{[ a]},x|\lambda) \leq C_{ a,b}   G(1_{[b]},x|\lambda),\quad \forall x\in X^{+}.
\end{equation}
\end{lem}
\begin{proof}
Let $N=d_{graph}(b_{},a_{})$, the graph distance from $b_{}$ to $a$, and let $a_1,\dots, a_N$ a path from $a_1=b$ to $a_N=a_{}$. For every $x\in X^{+}$,
$$   G(L_\phi^N1_{[b]},x|\lambda) = \lambda^{N} \sum_{n\geq N}\lambda^{-n}L_\phi^n 1_{[b]}(x) \leq   \lambda^{N}G(1_{[b]},x|\lambda). $$
Therefore,
\begin{align*}
&  G(1_{[b]},x|\lambda)\\
&\geq    \lambda^{-N} G( L_\phi^{^N}1_{[b]},x|\lambda) \\
&= \sum_{n=0}^\infty \sum_{\substack{b_1, \dots,b_{n+N-1}}}\lambda^{-n-N}e^{\phi_{n+N}(bb_1,\dots,b_{n+N-1}x)}\\
&\geq  \sum_{n=0}^\infty \sum_{b_{N+1},\dots,b_{N+n-1}}\lambda^{-n-N}e^{\phi_{n+N}(ba_2,\dots, a_{N-1}ab_{N+1},\dots,b_{N+n-1}x)}\\
&\geq   \lambda^{-N}e^{\phi_N(ba_2,\dots, a_{N-1}ax_a)-\sum_{i=2}^N Var_i(\phi)}\\
&\cdot \sum_{n=0}^\infty \sum_{b_{N+1},\dots,b_{N+n-1}}\lambda^{-n}e^{\phi_{n}(ab_{N+1},\dots,b_{N+n-1}x)}\\
&=\lambda^{-N} e^{\phi_N(ba_2,\dots, a_{N-1}ax_a)-\sum_{i=2}^N Var_i(\phi)}   G(1_{[a]},x|\lambda)
\end{align*}
where all inner sums range over all admissible paths. Then, the inequalities in  Eq. (\ref{eq:green_func_bound}) hold with  $C_{a,b}= \lambda^{N}e^{-\phi_N(ba_2,\dots, a_{N-1}ax_a)+\sum_{i=2}^N Var_i(\phi)}$ and $c_{a,b} = C_{b,a}^{-1}$.
\end{proof}
\begin{lem}
\label{lemma:kernel_uniform_boundness}
Let $(X^{+},T)$ be a transitive locally compact one-sided TMS and let $\phi$ be a $\lambda$-transient potential function with summable variations.  Then, for every $f\in C_c^{+}(X^{+})$, $f\neq 0$, there exist $c_{f }, C_f>0$ s.t.
$$ c_f\leq   K(f,x|\lambda)\leq C_f, \quad \forall x\in X^{+}.$$
\end{lem}
\begin{proof}
Since $supp(f)$ is compact, there exist $b_1,\dots, b_N\in S$ with $supp(f) \subseteq  \cup_{i=1}^N [b_i]$. Then, by the linearity of the Ruelle operator and by Lemma \ref{lemma:kernel_boundness},
\begin{align*}
  K(f,x|\lambda)=& \frac{  G(f,x|\lambda)}{  G(1_{[o]},x|\lambda)}\\
\leq& \sum_{i=1}^N \frac{\max_{x'\in [b_i]}\{f(x')\}  G(1_{[b_i]},x|\lambda)}{  G(1_{[o]},x|\lambda)}\\
\leq & \sum_{i=1}^N {\max_{x'\in [b_i]}\{f(x')\}}C_{b_i,o}.
\end{align*}
Let $(b_1',\dots, b_{M}')$ an admissible word of length $M$ s.t. $c:=\min_{x'\in [b_1',\dots, b_{M}']} \{f(x')\}>0.$
Then, with $x_{b_M'}\in T[b_M']$,
\begin{align*}   G(f,x| \lambda) \geq& c  G(1_{[b_1',\dots, b_{M}']},x|\lambda)\\
\geq&c \sum_{k=M}^\infty \lambda^{-k}(L_\phi^k 1_{[b_1',\dots, b_{M}']})(x)\\
\geq & ce^{-\phi_M(b_1',\dots, b'_M x_{b'_M})-\sum_{j=2}^\infty Var_j(\phi) \ }\lambda ^{-M} \sum_{k=0}^\infty \lambda^{-k}(L_\phi^k 1_{[ b_{M}']})(x)\\
=&c  e^{-\phi_M(b_1',\dots, b'_M x_{b'_M})-\sum_{j=2}^\infty Var_j(\phi) }\lambda ^{-M}   G(1_{[b_M']},x|\lambda).
\end{align*}
By Lemma \ref{lemma:kernel_boundness}, 
\begin{align*}
  G(1_{[b_M']},x|\lambda)\geq c_{b_{M}' ,o}  G(1_{[o]},x|\lambda) 
\end{align*}
and the lemma follows.
\end{proof}
Let $S^*=\{w_i\}_{i\in \mathbb{N}}$ be an enumeration of all admissible finite words on $S$. We define a new metric $\rho$ on $X^{+}$ by$$ \rho(x,y|\lambda) = \sum_{i=1}^\infty \frac{|  K(1_{[w_i]},x|\lambda)-  K(1_{[w_i]},y|\lambda)|+|1_{[w_i]}(x)-1_{[w_i]}(y)|}{2^i(C_{1_{[w_i]}}+1)}$$ where $C_{1_{[w_i]}}$ is a constant as in Lemma \ref{lemma:kernel_uniform_boundness}. It is easy to verify that $\rho$ is indeed a metric.
\begin{defn}
  The \textit{$\lambda$-Martin compactification} of $X^{+}$ and $\phi$, denoted by $ \widehat{X^+}(\lambda)$, is the completion of $X^{+}$ w.r.t. $\rho$. 
The \textit{$\lambda$-Martin boundary} of $X^{+}$ and $\phi$ is $ {\mathcal{M}}(\lambda):=  \widehat{X^+} (\lambda)\setminus X^{+}$. \end{defn}
We  will often abuse the notations and write $\widehat {X^+}(\lambda) = \widehat {X^+}, {\mathcal{M}}(\lambda)= {\mathcal{M}}$. For  examples of non-trivial Martin boundaries in the probabilistic settings see \cite{sawyer_1997}   and also \cite{pieordeiio_1992} for examples of the dependence (or independence) of the Martin boundary with $\lambda$.
 The following proposition describes  the main  properties of $(\widehat{X^+} (\lambda), \rho)$. The proof is elementary and can be found in the appendix.
\begin{prop}
\label{claim:compactification_prop}
Let $(X^{+},T)$ be a transitive locally compact one-sided TMS and let $\phi$ be a $\lambda$-transient potential function with summable variations.  
\begin{enumerate}[label={(\arabic*)}]
\item
\label{claim:compactification_prop_2}
Let $x_n,x\in X^+$. Then, 
$$x_n \xrightarrow[]{d} x \Longleftrightarrow x_n \xrightarrow[]{\rho} x. $$
\item
\label{claim:compactification_prop_escapes_compact}
If $x_n\xrightarrow[]{\rho}\omega \in \mathcal{  M}(\lambda)$ with $x_n\in X^{+}$, then $x_n\rightarrow \infty$ (escapes every $d$-compact set).

\item 
$\widehat{X^+} (\lambda)$ is compact and $\mathcal{M}(\lambda)$ is closed w.r.t. $\rho$.
 
\item 
Let $A\subseteq X^{+}$. Then, $A$ is $d$-open iff $A$ is $\rho$-open.

\item
For every $f\in C_c(X^{+})$, $  K(f,\cdot|\lambda)\in C(X^+)$ and   can be extended uniquely to a continuous function on $\widehat{X^+} (\lambda)$.
\item
\label{claim:compactification_prop_6}
For every $x \in {\widehat{X^+} (\lambda)},$ $  K(\cdot, x|\lambda)$ viewed as a positive linear functional on $C_c(X^{+})$ defines an $\lambda$-excessive measure on $X^{+}$. If $\omega\in  {\mathcal{M}}(\lambda)$, then $  K(\cdot, \omega|\lambda)$ defines a $\lambda$-conformal measure.
\end{enumerate}
\end{prop}
\begin{defn}
\label{def:omega_measure}
Given $\omega\in \mathcal{M}(\lambda)$, let $\mu_\omega$ denote the measure $$\mu_\omega (f) := K(f,\omega|\lambda) = \lim_{x\xrightarrow[]{\rho}\omega}K(f,x|\lambda). $$
\end{defn}
By Proposition \ref{claim:compactification_prop}, $\mu_\omega$ is a $\lambda$-conformal measure.  
Since $X^{+}$ is not compact but $\widehat {X^+}(\lambda)$ is compact, the boundary  is not empty and  Theorem \ref{thm:conf_existence} follows. In fact, the conformal measure constructed by Cyr in \cite{cyr_2010} to prove Theorem \ref{thm:conf_existence} is of the form $  K(\cdot, \omega|\lambda)$ for some $\omega\in \mathcal{  M}(\lambda)$.  The assumption that $|S|=\infty$ is crucial; otherwise $X^+$ is compact and the boundary is empty. In the next section, we show that all extremal conformal measures correspond to  boundary points.

\section{Integral representation of eigenmeasures}
\label{section:integration}
 In this section we describe all positive $\lambda$-eigenmeasures using an integral formula:
\begin{thm}
\label{thm:boundary_rep_theorem}
Let $(X^{+},T)$ be a transitive locally compact one-sided TMS and let $\phi$ be a $\lambda$-transient potential function with summable variations. Then, for every
$\mu\in \conf(\lambda)$,  there exists a finite measure $\nu$ on $\mathcal{M}(\lambda)$ s.t. for every $f\in C_c(X^{+})$, 
\begin{equation}
\label{eq:boundary_representation}
\mu(f) = \int_{\omega \in\mathcal{M}(\lambda)}K^{}(f,\omega|\lambda)d\nu(\omega).
\end{equation}
\end{thm}
Later on in this section, we will introduce the \textit{minimal Martin boundary} which yields a unique representation.  

To prove Theorem \ref{thm:boundary_rep_theorem} we need the following lemmas.
\begin{lem}
\label{lemma:riesz}
Let $(X^{+},T)$ be a transitive locally compact one-sided TMS and let $\phi$ be a $\lambda$-transient potential function with summable variations.  Let $\mu$ be a $\lambda$-excessive measure. 
Then the limit $\mu^* = \lim_{n\rightarrow\infty} \lambda^{-n}(L_\phi^*)^n \mu$ exists, $L_\phi^* \mu^* = \lambda\mu^*$ and if $\nu = \mu - \lambda^{-1}L_\phi^* \mu$, then 
$$\mu(f) = \int G(f,x|\lambda)d\nu(x) + \mu^*(f), \quad \forall f\in C_c(X). $$
\end{lem}
\begin{proof}
We write
$$ \mu = \sum_{k=0}^{n-1}\lambda^{-k}(L_\phi^*)^k\left( \mu - \lambda^{-1}L_\phi^*\mu  \right) +\lambda^{-n} (L_\phi^* )^n \mu.$$
Since $\mu$ is $\lambda$-excessive, the sequence  $\lambda^{-n}((L_\phi^* )^n \mu)(f)$ is non-increasing for every $f\in C_c^+(X^+)$ and the limit $\lim_{n\rightarrow\infty}\lambda^{-n}(L_\phi^* )^n \mu$ is a well-defined positive linear functional on $C_c(X^+)$.  As such, it defines a non-negative Radon measure on $X^+$ which is obviously conformal. In particular, the infinite sum $$\sum_{k=0}^{\infty}\lambda^{-k}(L_\phi^*)^k\left( \mu - \lambda^{-1}L_\phi^*\mu  \right)(f)$$ converges for all $f\in C_c(X^+)$ and the lemma follows. 
\end{proof}

\begin{defn}
 The minimum $\mu_1\wedge \mu_2$  of two measures $\mu_1, \mu_2$ is defined by $$\frac{d(\mu_1\wedge \mu_2)}{d(\mu_1+ \mu_2)} = \min\left\{ \frac{d\mu_{1}}{d(\mu_1+ \mu_2)},\frac{d\mu_{2}}{d(\mu_1+ \mu_2)} \right\} \quad (\mu_1+\mu_2 )-a.e.$$
\end{defn}
\begin{lem}
\label{lemma:min_measures}
Let $(X^{+},T)$ be a  locally compact one-sided TMS.  Then, for every two positive Radon measure $\mu_1,\mu_2$ on $X^+$,
$$(\mu_1\wedge \mu_2)(f) \leq \min\{\mu_1(f), \mu_2(f)\}, \quad \forall f\in C_c(X^+).$$ 
If $\mu'$ is a positive Radon measure such that $\mu'(f)\leq \min\{\mu_1(f), \mu_2(f)\}$ for every $f\in C_c^+(X^+)$, then $\mu' \leq \mu_1\wedge \mu_2$.
\end{lem}
\begin{proof}
Let $\mu = \mu_1 \wedge \mu_2$. For every $f\in C_c^{+}(X^{+})$ and for $i=1,2$
\begin{align*}
\mu(f) =&\int f(x)  \min\left\{ \frac{d\mu_{1}}{d(\mu_1+ \mu_2)},\frac{d\mu_{2}}{d(\mu_1+ \mu_2)} \right \}(x)d(\mu_1+\mu_2)(x)\\
\leq& \int f(x)  \frac{d\mu_i}{d\mu_1+d\mu_2}(x) d(\mu_1+\mu_2)(x)=\mu_i(f)
\end{align*}
and  $\mu\leq \mu_i$. 

Let $\mu'$ be a positive Radon measure with $\mu' (f) \leq \min_{i}\{\mu_i(f)\}$. Then, 
\begin{equation}
\label{eq:min_measures1}\mu'(f) \leq \int f(x) d\mu_i(x) =  \int f(x) \frac{d\mu_i}{d(\mu_1 + \mu_2)}(x) d(\mu_1 + \mu_2)(x). 
\end{equation}
Inequality (\ref{eq:min_measures1}) extends to all non-negative measurable functions. By decomposing 
\begin{equation}
\label{eq:min_f_decomp}
f=1_{\left[  \min\left\{ \frac{d\mu_{1}}{d(\mu_1+ \mu_2)},\frac{d\mu_{2}}{d(\mu_1+ \mu_2)} \right \}=\frac{d\mu_{1}}{d(\mu_1+ \mu_2)}\right]}f+1_{\left[  \min\left\{ \frac{d\mu_{1}}{d(\mu_1+ \mu_2)},\frac{d\mu_{2}}{d(\mu_1+ \mu_2)} \right \}<\frac{d\mu_{1}}{d(\mu_1+ \mu_2)}\right]}f
\end{equation}
we obtain
that
$$ \mu'(f) \leq  \int f(x) \min\left\{ \frac{d\mu_{1}}{d(\mu_1+ \mu_2)},\frac{d\mu_{2}}{d(\mu_1+ \mu_2)} \right \} (x)d(\mu_1 + \mu_2)(x)=\mu(f).$$
\end{proof}
\begin{cor}
\label{cor:min_measures}
Let $(X^{+},T)$ be a  locally compact one-sided TMS and let $\phi$ be a potential function. If $\mu_1$ and $\mu_2$ are $\lambda$-excessive, then $\mu_1\wedge\mu_2$ is $\lambda$-excessive.
\end{cor}
\begin{proof}
By Lemma \ref{lemma:min_measures}, for every $f\in C_c^+(X)$,  
$$(\mu_1 \wedge \mu_2) (L_\phi f) \leq \min\{\mu_1(L_\phi f), \mu_2(L_\phi f)\}\leq \lambda \min\{\mu_1(f), \mu_2(f)\} $$
and  $(\mu_1 \wedge \mu_2) (L_\phi f)  \leq \lambda(\mu_1 \wedge \mu_2)(f)$.
\end{proof}
\begin{proof}[Proof of Theorem \ref{thm:boundary_rep_theorem}]
Assume w.l.o.g. that $\lambda=1$. Let $W_{n}$ be a sequence of compact sets increasing to $X^{+}$ and let 
$$ \eta_n (f) := \int_{W_n} G(f,x)d\mu(x).$$
Clearly $\eta_n$ are positive Radon measures. Consider $\mu_n = \mu \wedge \eta_n$.  Let $f\in C_c^{+}(X)$ and let $N$ be an integer large enough  so that $supp(f) \subseteq W_n$ for all $n\geq N$. Let $\mu' = \mu_{|supp(f)}$, the restriction of $\mu$ to $supp(f)$. Clearly $\mu'\leq \mu$ and $\mu'(f)=\mu(f)$. Moveover, for every $g\in C_c^+(X^{+})$ and $n\geq N$, 
\begin{align*} \mu'(g)=&\mu(g1_{supp(f)}) \\
\leq & \int_{W_n}g(x) d\mu(x)\\
\leq&   \int_{{W_{n}}}\sum_{k=0}^\infty(L_\phi^k g)(x) d\mu(x)=\eta_n(g). \end{align*}
Thus, by Lemma \ref{lemma:min_measures}, $\mu'\leq \mu_n$ and $\mu(f)= \mu'(f) \leq \mu_n(f)$ for all $n$ large enough. Since $\mu_n \leq \mu$ for all $n$, we have that $\lim_{n\rightarrow\infty}\mu_n(f)=\mu(f)$ for all $f\in C_c(X^+)$. 

 Observe that $\forall f\in C_c^+(X^+)$,
$$\ (L_\phi^* \eta_n) (f)  = \int_{W_n} \sum_{k\geq 1}(L_\phi^k f)(x)d\mu(x) \leq \int _{W_n}G(f,x)d\mu(x) = \eta_n(f),$$
and thus  $\eta_n$ is excessive. By Corollary \ref{cor:min_measures}, $\mu_n$ is excessive as well.  By the bounded convergence theorem, for every fixed $n$ and every $f\in C_c^+(X^+)$, $$((L_\phi^*)^m\mu_n)(f)\leq((L_\phi^*)^m\eta_n )(f)= \int_{W_n} \sum_{k\geq m}(L_\phi^kf)(x)d\mu(x)\xrightarrow[m\rightarrow\infty]{}0 .  $$  Then, by Lemma \ref{lemma:riesz}  there exists a measure $\nu_n$ on $\widehat{X^+}$ s.t. for every  $f\in C_c(X^{+})$, 
$$\mu_n(f) = \int G(f,x)d\nu_n(x) = \int K(f,x) G(1_{[o]},x)d\nu_n(x). $$
Let $d\nu_n'(x) = G(1_{[o]},x)d\nu_n(x)$. Since $\widehat{X^+}$ is compact and $\nu'_n(1) = \mu_n(1_{[o]})\leq \mu(1_{[o]})$ for all $n$, there exists a weak$^*$-converging subsequence $\nu'_{n_k}$ to a measure $\nu$.
Since $K(f,x)$ is $\rho$-continuous for every $f\in C_c(X^{+})$,
$$ \mu(f) = \int K(f,x)d\nu(x).$$ 

Next, we show that that $supp(\nu) \subseteq \mathcal{M}$. According to Proposition \ref{claim:compactification_prop}, for every $x\in \widehat {X^+}$,  
$$K^{ }(L_\phi f,x) \leq  K^{ }(f,x).$$
Moreover, for every $x\in X^{+}$ and $a\in S$,
$$K^{ }(1_{[a]},x) - K^{ }(L_\phi 1_{[a]},x) =\frac{1_{[a]}(x)}{G(1_{[o]},x)} .$$
Since $G(1_{[a]},x)$ is positive and continuous, it is bounded away from zero on any compact set $[a]$. Hence, 
\begin{align*}
0= \mu(1_{[a]}) - L_\phi^*\mu(1_{[a]})=& \int _{\widehat{ X^+}}\left(K^{ }(1_{[a]},x) - K^{ }(L_\phi 1_{[a]},x)_{ } \right) d\nu(x)\\
\geq & \int _{[a]}\left(K^{ }(1_{[a]},x) - K^{ }(L_\phi 1_{[a]},x)_{ } \right) d\nu(x)\\
= & \int_{[a]}G(1_{[o]},x)^{-1}d\nu(x)\\
\geq & \min_{x\in [a]}\{G(1_{[o]},x)^{-1}\}\nu([a])
\end{align*}
 whence $\nu([a])=0$, which implies that $\nu(X)=\sum_{a\in S}\nu([a])=0$.
\end{proof}
\begin{remk}
If we assume that $\mu$ is only $\lambda$-excessive, then we obtain similar results except that $\nu$ may charge $X^+$. 
\end{remk}

Next, we study the extremal points of the cone $\conf(\lambda)$.

\begin{defn}A measure  $\mu\in \conf(\lambda) $ is called \textit{extremal} if for every $\mu_1,\mu_2 \in \conf(\lambda)$ with $\mu = \mu_1+\mu_2$ we have that $\mu_i \propto\ \mu$. \end{defn}
Recall the definition of the measure $\mu_\omega$; $\mu_\omega (f)= K(f,\omega|\lambda)
$, for $\omega\in \mathcal{M}(\lambda)$ and $f\in C_c(X^+)$. Observe that if $\omega_1,\omega_2\in \mathcal{M(\lambda)}$ with $K(1_{[w]},\omega_1)= K(1_{[w]},\omega_2)$ for all $w\in S^*$ then $\rho(\omega_1,\omega_2)=0$, whence $\omega_1=\omega_2$.  Therefore, for all $\omega_1\neq \omega_2$, $\mu_{\omega_1}\neq \mu_{\omega_2}$ and vice versa. Moreover,  since  $\mu_{\omega_{i}}([o])=1$, $\mu_{\omega_1}$ and $\mu_{\omega_2}$ are not proportional.
\begin{lem}
\label{lemma:extremal_points}
Let $(X^{+},T)$ be a transitive locally compact one-sided TMS and let $\phi$ be a $\lambda$-transient potential function with summable variations.  Let $\mu\in \conf(\lambda)$ and assume that $\mu$ is extremal. If $\mu = \int_{\mathcal{M}}\mu_\omega d\nu(\omega)$ where $\nu $ is a positive and finite Borel measure on $\mathcal{M}$,  then there exists $\omega\in \mathcal{M}(\lambda)$ s.t. $\nu \propto \delta_\omega$.
\end{lem}
\begin{proof}
Assume w.l.o.g. that $\mu([o])=1$ and that $\lambda=1$. We follow the proof of Theorem 24.8 in \cite{woess_2000}.
  Let $\nu$ be a positive finite measure on $\mathcal{M}$ s.t.$$ \mu(\cdot) = \int K^{}(\cdot, \omega) d\nu(\omega).$$
We show that $\nu \propto\ \delta_\omega$ for some $\omega\in \mathcal{M}$. Then, since $\mu([a]) = K([a], \omega)=1$, $\nu = \delta_\omega$.
Assume that there exists a Borel set $B \subseteq \mathcal{M}$ s.t. $0<\nu(B) <\nu(\mathcal{M})$.  We define
 
$$ \mu_1  = \frac{1}{\nu(B)}\int_{\omega\in B} K(\cdot, \omega)d\nu(\omega)$$
and
$$ \mu_2  = \frac{1}{\nu(\mathcal{M} \setminus B)}\int_{\omega\in \mathcal{M}\setminus B} K(\cdot, \omega)d\nu(\omega).$$
Observe that $\mu_1$ and $\mu_2$ are conformal measures as well since $K(L_\phi f, \omega) = K(f,\omega)$ for every $f\in C_c(X^+)$ and $\omega\in \mathcal{M}$.
 Since 
$$\mu =\nu(B) \mu_1 + \nu(\mathcal{M} \setminus B)\mu_2$$ and $\mu$ is extremal, we have that $\mu_1 ,\mu_2\propto \mu$. Since $\mu_1([o]) = \mu_2([o])=\mu([o])$, we have that $\mu_1=\mu_2=\mu$. However, $\mu_1$ and $\mu_2$ charge $B$ and $\mathcal{M}\setminus B$ respectively, which contradicts the existence of the set $B$. In particular, for all Borel $B \subseteq \mathcal{M}$, $\nu(B) \in \{0, \nu(\mathcal{M})\}$, whence 
$\nu \propto\ \delta_\omega$, for some $\omega\in \mathcal{M}$.
\end{proof}
\begin{cor}
\label{cor:extremal_points}
Let $(X^{+},T)$ be a transitive locally compact one-sided TMS and let $\phi$ be a $\lambda$-transient potential function with summable variations.  Let $\mu\in \conf(\lambda)$ and assume that $\mu$ is extremal. Then,  then there exists $\omega\in \mathcal{M}(\lambda)$ s.t. $\nu \propto \delta_\omega$.
\end{cor}
\begin{proof}
The corollary follows by Theorem \ref{thm:boundary_rep_theorem} and Lemma \ref{lemma:extremal_points}.
\end{proof}
\begin{defn}
The \textit{$\lambda$-Minimal Martin Boundary} $\mathcal{M}_m(\lambda) $ is the set of all  points $\omega\in  \mathcal{M}(\lambda)$ with $\mu_\omega$ extremal in $\conf(\lambda)$.  
\end{defn}
The minimal boundary is a Borel set, see Proposition \ref{prop:minimal_borel} in the appendix. Corollary  \ref{cor:extremal_points} yields that $\{c \mu_\omega : \omega \in \mathcal{M}_m(\lambda), c \geq0\}$ is in fact the collection of all extremal points of $\conf(\lambda)$.
\begin{cor}
\label{cor:ind_of_ref_fucntion}
Let $(X^{+},T)$ be a transitive locally compact one-sided TMS and let $\phi$ be a $\lambda$-transient potential function with summable variations.  Then, the $\lambda$-minimal Martin boundary does not depend on the choice of the origin state $o$, namely if $\mathcal{M}_m^1(\lambda)$ and $\mathcal{M}_m^2(\lambda)$ are the minimal boundaries which correspond to origin points $o_1$ and $o_2$ respectively, then for every $\omega_1\in \mathcal{M}_m^1(\lambda)$ there exists a unique $\omega_2 \in \mathcal{M}_m^1(\lambda)$ s.t. $\mu^{1}_{\omega_1}\propto \mu^{2}_{\omega_2}$ where $\mu_{\omega_{i}}^i = \lim_{x\rightarrow \omega_i}\frac{K(\cdot, x)}{K(1_{[o_i]},x)}$.
\end{cor}
\begin{proof}
For every $i$, every point of the minimal Martin boundary w.r.t. an origin point $o_{i}$  corresponds to an extremal conformal measure (up to a constant) and vice versa. Therefore the choice of the origin state only affects the normalizing factor.\end{proof} 

It is natural to ask if the measure $\nu$ in Theorem \ref{thm:boundary_rep_theorem} is unique. In general, $\nu$ may be non-unique. However,  if we restrict the support of $\nu$\ to $\mathcal{M}_m(\lambda)$, then we do obtain uniqueness.
\begin{thm}
\label{thm:minimal_kernel_rep}
Let $(X^{+},T)$ be a transitive locally compact one-sided TMS and let $\phi$ be a $\lambda$-transient potential function with summable variations. Let $\mu\in \conf(\lambda)$. Then, there exists a  \textbf{unique} finite measure $\nu$ on $\mathcal{M}_m(\lambda)$ s.t.
$$\mu(f) = \int_{\mathcal{M}_m(\lambda)}K(f,\omega|\lambda)d\nu(\omega). $$
\end{thm}
\begin{proof}
Assume w.l.o.g. that $\lambda=1$.
To prove the theorem, we first recall the notion of a lower semilattice.
\begin{defn}
\label{defn:lattice}
Let $V$ be a topological vector space 
 with a partial ordering $\leq$. An element $u\in V$ is the called the \textit{meet} of two elements $v_1\in V$ and $v_2\in V$  if for every $w\in V$ with $w\leq v_i$ ($i=1,2$) we have that $w\leq u$. The vector space $V$ is called a \textit{lower semilattice} if every two elements have a meet. For more on on this subject, see \cite{clifford_1961}.\end{defn}
Our aim is to apply the following version of Choquet's theorem;
\begin{thm}[Furstenberg \cite{furstenberg_1965}]
Let $V$ be a weak$^*$-closed cone of positive measures on a separable, locally compact space and let $\mathcal{E}$ denote the extremal rays of $V$. Suppose that there is a positive function of compact support $\psi$ with $\mu(\psi)>0$ for all $\mu\in V$, $\mu\neq 0$, and let $V_1=\{\mu \in V: \mu(\psi)=1\}$. Then, for each $\mu\in V$ there exists a measure $\nu$ on a Borel subset $B\subseteq\mathcal{E} \cap V_1$ s.t. $\mu = \int_{B}\mu' d\nu(\mu')$.
If $V$ is a lower semilattice, then the measure $\nu$ is unique.
\end{thm}
It is easy to verify that $\conf$ is weak$^*$-closed. By Lemma \ref{lemma:positive_on_cylinders},   $\mu(1_{[o]})>0$ for every $\mu\in \conf$. Thus, to obtain  uniqueness it suffices to show that $\conf$ is a lower semilattice.

Let $\mu_1, \mu_2 \in \conf$ and let $\mu'=\mu_1 \wedge \mu_2$. By Corollary \ref{cor:min_measures}, $\mu'$ is excessive and the limit measure
$\mu'' = \lim_{n\rightarrow\infty}(L_\phi^*)^n \mu'$ exists.  Clearly $\mu'' \geq 0$. Since $L_\phi f\in C_c^{+}(X^{+})$  for all $f\in C_c^+(X^+)$,
$$ (L_\phi^* \mu'') (f) = \mu''(L_\phi f) = \lim_{n\rightarrow\infty}((L_\phi^*)^n \mu' )(L_\phi f) =\lim_{n\rightarrow\infty}((L_\phi^*)^{n+1} \mu' )( f)=\mu''(f) $$ 
and $\mu'' \in \conf$.

Suppose $\mu_3 \in \conf$ with $\mu_3 \leq \mu_i$, $i=1,2$. Then,  by Lemma \ref{lemma:min_measures}, $\mu_3 \leq \mu'$. Since $L_\phi^* \mu_3 = \mu_3$, we have that $\mu_3 \leq (L_\phi^*)^n \mu'$ for every $n$, which implies that $\mu_3 \leq \mu''$ and $\conf$ is indeed a lower semilattice. 
\end{proof}

\section{Convergence to the boundary}
\label{section:convergence}
We saw in the previous section that every $\lambda$-conformal measure $\mu$ can be uniquely presented in the form 
$$\mu = \int_{\mathcal{M}_m(\lambda)}\mu_\omega d\nu(\lambda) $$
for some boundary measure $\nu$ on the minimal Martin boundary. In this section we show that $\nu$ is determined by the $\mu$-almost sure behaviour of $T^nx$ as $n\rightarrow\infty$:
\begin{thm}
\label{thm:conv_to_boundary}
Let $(X^{+},T)$ be a transitive locally compact one-sided TMS and let $\phi$ be a $\lambda$-transient potential function with summable variations. Let $\mu\in \conf(\lambda)$. Then, 
\begin{enumerate}
\item 
For $\mu$-a.e. $x\in X^{+}$, the $\rho$-limit $\lim_{n\rightarrow\infty}T^nx$ exists and belongs to $\mathcal{M}_m(\lambda)$.
\item
If $\mu = \mu_\omega$ for some $\omega\in \mathcal{M}_m(\lambda)$, then for $\mu_\omega$-a.e. $x\in X^{+}$, $T^nx \rightarrow\omega$.
\item
For every Borel set $E \subseteq \mathcal{M}(\lambda)$,
\begin{equation}
\label{eq:nu_prop}
\nu(E) = \mu\{x\in [o] : \exists \omega \in E \text{ s.t. }T^nx \xrightarrow[]{\rho}\omega\}. 
\end{equation}
\end{enumerate}
\end{thm}

When $\phi$ is $\lambda$-transient, all $\lambda$-conformal measures admit a dissipative behaviour.
For a set $F \subseteq X^{+}$ let 
$$F_\infty = \{x\in X^{+} : T^nx \in F \text{ for infinitely many } n\geq0\}.$$
\begin{prop}
\label{prop:dissipative_on_compacts}
Let $(X^{+},T)$ be a transitive locally compact one-sided TMS and let $\phi$ be a $\lambda$-transient potential function with summable variations.  Let $F \subseteq X^{+}$ be a compact set. Then, for every $\mu\in \conf(\lambda)$, $\mu(F_{\infty})=0$.\end{prop}
\begin{proof}
This result is well-known, see \cite{aaronson_1997}. For completeness, we provide a proof in the appendix.
\end{proof}
For a measure $\mu$ and a set $F \subseteq X^{+}$ let
$\mu_{|F}(f) = \mu(f \cdot 1_F)$,
the restriction of $\mu$ to $F$.
A measure $\mu$ on $X^+$ (possibly non-invariant or infinite) is said to be
\textit{ergodic} if for every measurable set $A\subseteq X^+$ with $T^{-1}A=A$, $\mu(A)=0$ or $\mu(X^+\setminus A)=0$. 
\begin{lem}
\label{lemma:erogidicty_of_extremal_points}
Let $(X^{+},T)$ be a transitive locally compact one-sided TMS and let $\phi$ be a $\lambda$-transient potential function with summable variations.  Let $\omega \in \mathcal{M}_m(\lambda)$. Then, $\mu_\omega$ is ergodic.
\end{lem}
\begin{proof}
We write $\mu=\mu_\omega$. Let $A$ be a $T$-invariant set, with $T^{-1}A = A$.   
Then, 
\begin{align*}\mu_{|A}(L_\phi 1_{E})=&  \int (L_\phi1_E )(x) 1_A(x)d\mu(x) \\
=& \int (L_\phi\left(1_E \cdot 1_A\circ T\right))(x)d\mu(x)\\
=& \int (L_\phi(1_E \cdot 1_A))(x)d\mu(x)\\
=&\lambda  \int 1_E(x) \cdot 1_A(x)d\mu(x)=\lambda\mu_{|A}(E).
\end{align*}
Similarly, $L_\phi^* \mu_{|A^c}=\lambda\mu_{|A^c}$. Clearly $\mu_{|A}, \mu_{|A^c}$ are Radon, and thus $\mu_{|A},\mu_{|A^{c}} \in \conf(\lambda)$. Since $\mu = \mu_{|A}+\mu_{|A^c}$ and $\mu$ is extremal, we must have that $\mu_{|A},\mu_{|A^{c}}\propto \mu$. Since $\mu_{|A}$ and $\mu_{|A^{c}}$ are mutually singular,  $\mu_{|A}\equiv0$ or $\mu_{|A^c}\equiv0$ which implies that either $\mu(A)=0$ or $\mu(A^c)=0$.
\end{proof}
\begin{remk}
Observe that $\mu_\omega$ is not necessarily $T$-invariant. In fact, $\mu_\omega$ is $T$-invariant iff $L_\phi 1(x)=1$ for $\mu_\omega$-a.e. $x$, see \cite{keane_1972}.
One can easily ``fix" $\mu_\omega$ to be $T$-invariant with a positive eigenfunction as a density. For more on the positive eigenfunctions see Section \ref{section:harmonic_functions}.
\end{remk}
For a set $F \subseteq X^{+}$ let 
$$F_{+}= \{x\in X^{+} : \exists n\geq0 \text{ s.t. }T^nx \in F\}.$$
We denote by $\overline{F}$  the topological closure of $F$ in $(\widehat {X^+},\rho)$.
\begin{lem}
\label{lemma:restricted_integral}
Let $(X^{+},T)$ be a transitive locally compact one-sided TMS and let $\phi$ be a $\lambda$-transient potential function with summable variations.  Let $\mu \in \conf(\lambda)$ and let $F\subseteq X^{+}$ be a Borel set. Then, there exists a finite measure $\nu$ with $supp(\nu) \subseteq \overline{F}$  s.t.
$$\mu_{|F_+}(f) = \int K(f,x|\lambda)d\nu(x), \quad \forall f\in C_c(X^{+}). $$
\end{lem}
\begin{proof}
Assume w.l.o.g. that $\lambda=1$. Assume first that $F$ is compact.
Observe that if $Tx\in F_{+}$ then $x\in F_{+}$ as well, whence $T^{-1}F_{+} \subseteq\ F_{+}$. In particular, for every $f\in C_c^+(X^{+})$, 
\begin{align*} L_{\phi}^*(\mu_{|F_+})(f) =& \int (L_\phi f)(x)1 _{F_+}(x)d\mu(x) \\
=& \int L_\phi (f \cdot (1 _{F_+}\circ T))(x)d\mu(x)\\
=&   \int f(x) \cdot( 1 _{F_+}\circ T)(x)d\mu(x)=\mu_{|T^{-1}F_+}(f) 
\leq\mu_{|F_+}(f). \end{align*}
 Hence $\mu_{|F_+}$ is excessive and by Lemma \ref{lemma:riesz} we can write
\begin{align*}\mu_{|F_+} (f) =& \int_X G(f,x) d(\mu_{|F_+} - L_\phi^* (\mu_{|F_+}))(x) +  \lim_{n\rightarrow\infty}((L_\phi^n)^*(\mu_{|F_+}))( f).  \end{align*}
 Observe that $\bigcap_{n}  T^{-n}F_+ = F_\infty$.  By the bounded convergence theorem, for every compact set $K\subseteq X^+$$$ ((L_\phi^n)^*(\mu_{|F_+}))(K) = \mu\left(  T^{-n}F_+\cap K\right)\xrightarrow[n\rightarrow\infty]{} \mu(F_\infty \cap K). $$ 
By Proposition  \ref{prop:dissipative_on_compacts}, $\mu(F_\infty \cap K)=0$, which leads to  $$\mu_{|F_+} (f)=  \int_{\widehat{X^+}} K(f,x) d\nu(x)$$
with $d\nu(x) = G(1_{[o]},x) d(\mu_{|F_+} - L_\phi^* (\mu_{|F_+}))(x)$. Since $F_+ \setminus T^{-1}F_+ \subseteq\ F$, 
$$supp(\nu) = supp(\mu_{|F_+} - L_\phi^* (\mu_{|F_+})) = supp(\mu_{|F_+\setminus T^{-1}F_+}) \subseteq F. $$

Now, assume that  $F$ is arbitrary and let $F_m $ be compact increasing sets s.t. $\mu(F \setminus \cup_m F_m)=0$. Let $\nu_m$  be a measure with $supp(\nu_m) \subseteq  F_m$ and 
$$ \mu_{|(F_{m})_+} (\cdot)= \int K(\cdot, x)d\nu_m(x).$$ 
Let $g\in C_c^{+}(X^{+})$ with $\mu(g)>0$. By Lemma \ref{lemma:kernel_uniform_boundness} there exists $c_g>0$ s.t. $K(g,x) \geq c_g$, $\forall x\in \widehat {X^+}$.  Then, 
$$\nu_m( 1) \leq c_g^{-1}\int  K(g,x)d\nu_m(x)=c_{g}^{-1} \mu_{|(F_m)_+}(g) \leq c_g^{-1}\mu(g) <\infty$$  and the measures $\nu_m$ are uniformly bounded. Working in the compact space $\widehat{X^+}$, we can find a weak${}^*$ converging sub-sequence $\nu_{m_k}\rightarrow\nu$. Since $(F_m)_+ \nearrow F_{+}$ and $K(f, x)$ is $\rho$-continuous on $\widehat {X^+}$, we have that 
$$\mu_{|F_+}(f) =  \int K(f, x)d\nu(x), \quad \forall f\in C_c(X^{+}). $$ 
Lastly, since  $supp(\nu_{m_k}) \subseteq F_{m_{k}}\subseteq F$ then $supp(\nu) \subseteq\ \overline{F}$ (the closure in $\widehat{X^+})$.
\end{proof}
\begin{lem}
\label{lemma:rep_support_closure}
Let $(X^{+},T)$ be a transitive locally compact one-sided TMS and let $\phi$ be a $\lambda$-transient potential function with summable variations. Let $\mu \in \conf(\lambda)$ and let $F\subseteq X^{+}$ be a Borel set. Then, there exists a finite measure $\nu$ with $supp(\nu) \subseteq\ \overline{ F}\cap \mathcal{M}(\lambda)$ s.t.
$$\mu_{|F_\infty}(f) = \int K(f,x)d\nu(x), \quad \forall f\in C_c(X^{+}).$$
\end{lem}
\begin{proof}
Assume w.l.o.g. that $\lambda=1$. Recall that $T^{-n}F_+$ is a non-increasing sequence of sets and that $\bigcap_n T^{-n}F_+ = F_\infty$. Thus, by the dominated convergence theorem, for every $f\in C_c^+(X^+)$
$$((L_\phi^*)^n\mu_{|F_+})(f)  = \mu_{|T^{-n}F_+}(f) = \mu(f \cdot 1_{_{T^{-n}F_{+}}})\rightarrow \mu(f \cdot 1_{F_\infty}) = \mu_{|F_\infty}(f). $$
Observe that for every $x\in X^{+}$ and
 every $f\in C_c^+(X^+)$$$K(L_{\phi}^nf,x)=\frac{\sum_{k\geq n}L_\phi^k f(x)}{G(1_{[o]},x)}\searrow0.$$ Moreover, for every $\omega\in \mathcal{M}$, $K(L_{\phi}^nf,\omega)=K(f,\omega)$ (see Proposition \ref{claim:compactification_prop}). By  Lemma \ref{lemma:restricted_integral} and the bounded convergence theorem, we obtain that 
\begin{align*}\mu_{|F_\infty}(f) =& \lim_{n\rightarrow\infty}((L_\phi^*)^n(\mu_{|F_+}))(f) \\
=&\lim_{n\rightarrow\infty} \int_{\overline{F}}K(L_{\phi}^nf,x)d\nu(x)\\
=&  \int_{\overline{F} \cap \mathcal{M}} K(f,\omega)d\nu(\omega).
\end{align*}
\end{proof}

\begin{proof}[Proof of Theorem \ref{thm:conv_to_boundary}]
Assume w.l.o.g. that $\lambda=1$. Assume first that $\mu = \mu_\omega $ for some $\omega \in \mathcal{M}_m$. 
For $\epsilon>0$ let 
$$F_\epsilon = \{x\in X^{+}:\rho(x, \omega)  \geq \epsilon \}$$
and let 
$$A_\epsilon = (F_\epsilon)_\infty = \{x\in X^{+} : T^nx \in F_\epsilon \text{ for infinitely many } n \} .$$
Clearly $T^n x\rightarrow\omega$ iff $x\not\in A_\epsilon$ for every $\epsilon>0$. Thus it suffices to show that $\mu_\omega(A_\epsilon)=0$, for every $\epsilon>0$. 

Since $A_\epsilon$ is a $T$-invariant set and $\mu_\omega$ is an extremal measure, we have by Lemma \ref{lemma:erogidicty_of_extremal_points} that either $\mu_\omega(A_\epsilon)=0$ or $\mu_\omega(A_\epsilon^c)=0$ and  we only have to exclude the second case. Assume that $\mu_\omega(A_\epsilon)\neq 0$. Then $\mu_\omega = \mu_{\omega | A_\epsilon}$.  According to Lemma \ref{lemma:rep_support_closure} there exists a measure $\nu$ with $supp(\nu) \subseteq \overline{F_\epsilon} \cap \mathcal{M}$ s.t. 
$$\mu_\omega (\cdot)= \mu_{\omega | A_\epsilon}(\cdot) = \int K(\cdot,\omega') d\nu(\omega'). $$
Since $\mu_\omega$ is extremal, by Lemma \ref{lemma:extremal_points} we must have that $\nu \propto\ \delta_\omega$. This implies that $\omega \in \overline{F_\epsilon}$, which is clearly a contradiction. Hence $\mu_\omega(A_\epsilon)=0$.

Next, consider some arbitrary $\mu \in \conf$. Let 
$$A=\{x\in X^{+} : T^nx \text{ has no $\rho$-limit}\}.  $$
Let $\nu$ be the measure from Theorem \ref{thm:minimal_kernel_rep} s.t.
$$\mu = \int_{\mathcal{M}_m} \mu_\omega d\nu_{}(\omega). $$
Since $\mu_\omega(A)=0$ for every $\omega \in \mathcal{M}_m$, $\mu(A)=0$ as well. In particular,  for $\mu$-a.e. $x\in X^+$, $T^nx$ converges. By Proposition \ref{prop:dissipative_on_compacts}, $T^nx$ must converge to a boundary point. For all $\omega \in \mathcal{M}_m$, $$\mu_\omega (\left\{ x\in X^+ : \lim T^nx \in \mathcal{M}\setminus \mathcal{M}_m \right\})=0 $$
and therefore for $\mu$-a.e. $x\in X^+$,   $T^nx$  converges to a point in $ \mathcal{M}_m$. 

As for Eq. (\ref{eq:nu_prop}), since for every $\omega\in \mathcal{M}_m$, $$1 = \mu_\omega([o]) = \mu_
\omega([o] \cap \{x\in X^+: T^nx\rightarrow\omega\}) $$
we have that
\begin{align*}
\nu(E) = & \int_E 1d\nu(\omega)\\
=&\ \int_E  \mu_
\omega\left([o] \cap \{x\in X^+: T^nx\rightarrow\omega\}\right)d\nu(\omega)\\
=& \int_{\mathcal{M}_m}  \mu_
\omega\left([o] \cap \{x\in X^+:\exists \omega\in E \text{ s.t. } T^nx\rightarrow\omega\}\right)d\nu(\omega)\\
=& \mu\left([o] \cap \{x\in X^+:\exists \omega\in E \text{ s.t. } T^nx\rightarrow\omega\}\right).
\end{align*}
\end{proof}  

\section{The reversed Martin boundary and positive eigenfunctions}
\label{section:harmonic_functions}
So far we have focused on positive $\lambda$-eigenmeasures. We now turn our attention to the positive $\lambda$-eigenfunctions, specifically to positive eigenfunctions with  uniformly continuous logarithm:
\begin{defn}
 Let 
$$\harm(\lambda)=\left\{f\in C(X^{+}) : f>0, L_\phi f=\lambda f \text{ and } \log f\text{ is uniformly continuous}\right\}.$$
\end{defn}
The uniform regularity condition in the definition of $\harm$ appears naturally when trying to represent  eigenfunctions in forms of Martin kernels. 

One possible approach to study the positive eigenfunctions is  via a direct construction of a suitable Martin boundary, as in the study of the eigenmeasures.
However, this approach tends to be technical, leads to redundant  proofs and does not establish any connection between the eigenfunctions and the eigenmeasures. Thus we take a different approach; studying the eigenmeasures on the negative one-sided TMS. In particular, we establish a $1-1$ correspondence between eigenfunctions on the positive one-sided TMS and eigenmeasures on the negative one-sided TMS.  For a similar duality  in the probabilistic settings,  see \cite{kemeny_2012} and also \cite{revuz_2008}.

We start with definitions.
 \begin{defn}Let
$$ X = \{z\in S^{\mathbb{Z}}:\mathbb{A}_{(z)_{i},(z)_{i+1}}=1, \forall i\in\mathbb{Z}\}$$
and let
$$ X^- = \{y\in S^{\mathbb{-N}\cup\{0\}}:\mathbb{A}_{(y)_{i},(y)_{i+1}}=1, \forall i<0 \}.$$
To avoid confusions, in this section points of $X^+$ will be denoted by $x$, points of $X^-$ will be denoted by $y$ and points of $X$ by $z$. 
The\textit{ (two-sided) left shift} $T:X\rightarrow X$ is the transformation $(Tz)_i = (z)_{i+1}$, the \textit{(two-sided) right shift} $T^{-1}:X\rightarrow X$ is the transformation $(T^{-1}z)_i = (z)_{i-1}$ and the \textit{(one-sided) right shift} $T^{-1}:X^-\rightarrow X^{-}$ is the transformation $(T^{-1}y)_i = (y)_{i-1}$.  Then, $(X^-, T^{-1})$ is the \textit{negative (one-sided) topological Markov shift} and $(X,T)$ is the \textit{two-sided topological Markov shift}. 

For $z\in X^+$, let $z^+$ be the projection of $z$ to $X^+$ 
$$z^+ = ((z)_{0}, (z)_{1},(z)_2,\dots) $$
and let $z^-$ the projection of $z$ to $X^-$
$$z^- = (\dots, (z)_{-2}, (z)_{-1},(z)_0). $$ The notion of cylinder sets, the metric $d$ and the variation of a function (Eq. (\ref{eq:def_d})-(\ref{eq:def_var})) can be naturally extended to $X$ and $X^-$ as well, by allowing negative indices. 

For a potential function $\phi^-:X^-\rightarrow\mathbb{R}$, the \textit{Ruelle operator} $L_{\phi^-}$ is
$$ (L_{\phi^-} f) (y) = \sum_{y':T^{-1}y'=y}e^{\phi^-(y')}f(y'), \quad f\in C(X^-), y\in X^-.$$
In particular, 
$$ (L_{\phi^-}^n f )(y) = \sum_{y':T^{-n}y'=y}e^{\phi_{n}^-(y')}f(y'), \quad f\in C(X^-), y\in X^-$$
where $\phi_n^-(y) = \sum_{i=0}^{n-1}\phi^-(T^{-i}y)$. 
\end{defn}
To simplify the notations we write  $T,T^{-1}$ and $d$ without stating which type of a TMS we acting on. The intention should be clear from the context. When handling a two-sided point $z\in X$, the zero entry may be marked with a dot over it, e.g. $z = (\dots, (z)_{-1}, \dot{(z)}_0, (z)_1,\dots)$. 

 Our approach to the problem is via the Martin boundary of the negative one-sided shift $X^-$, which we denote it by $\cev{ \mathcal{M}}$. In order for $\cev{ \mathcal{M}}$ to exist, we need $X^-$ to be  locally compact as well. For this reason we add assumption \ref{assumption:locally_compact_two_sides}. This, together with assumption \ref{assumption:locally_compact}, implies that  $\sum_{b}\mathbb{A}_{a,b}+\sum_b \mathbb{A}_{b,a}<\infty$ for all $a\in S$ and that $X$ is locally compact.  
Before stating and proving the main results of this section, we first handle the question which potential function should we equip $X^-$ in order to construct $\cev{ \mathcal{M}}$. The following propositions provides us with a natural one. \begin{defn}
Two potentials $\phi^+ :X^+\rightarrow\mathbb{R}$, $\phi^-:X^-\rightarrow\mathbb{R}$ are \textit{cohomologous via a transfer function} $\psi:X\rightarrow\mathbb{R}$ if
\begin{equation}
\label{eq:cohomo_eq}
\phi^+(z^+) -\phi^-(z^-) = \psi(z) - \psi(Tz), \quad \forall z\in X. 
\end{equation}
\end{defn}
\begin{prop}
Let $\phi^+:X^+\rightarrow\mathbb{R}$ be a potential function with summable variations. Then, there exists $\phi^-:X^-\rightarrow\mathbb{R}$ with summable variations and a uniformly continuous function  $\psi:X\rightarrow\mathbb{R}$ s.t. $\phi^+$ and $\phi^-$ are cohomologous via $\psi$. 
\end{prop}
\begin{proof}
See  \cite{bowen_1975,coelho_1998,daon_2013}.
\end{proof}
\begin{prop}
\label{claim:left_right_transience}
Let $(X,T)$ be a transitive locally compact two-sided TMS and let $\phi^+ :X^+\rightarrow\mathbb{R}$, $\phi^-:X^-\rightarrow\mathbb{R}$  be  two potential functions with summable variations and which are cohomologous via a uniformly continuous transfer function $\psi:X\rightarrow\mathbb{R}$. Then,  $\phi^+$ is $\lambda$-transient iff $\phi^-$ is $\lambda$-transient. 
\end{prop}
\begin{proof}
See appendix.
\end{proof}
In what follows, we assume that  $\phi^+:X^+\rightarrow\mathbb{R}$ and $\phi^-:X^-\rightarrow\mathbb{R}$ are  $\lambda$-transient  potential functions with summable variations which are cohomologous via a uniformly continuous  transfer function   $\psi:X\rightarrow\mathbb{R}$. 

Denote by $\cev {K}(\cdot, \cdot|\lambda)$ the Martin kernel w.r.t. $(X^- , \phi^-)$, by $\cev{ \mathcal{M}}(\lambda)$ the corresponding Martin boundary and by $\cev{ \mathcal{M}}_m(\lambda)$ the minimal boundary. We show that the eigenfunctions on the positive one-sided TMS are in fact equivalent, via a simple reduction, to the conformal measures of the negative one-sided TMS.   

\begin{defn}For a point $x\in X^+$, let $\psi_{x}:T^{-1}[(x)_0]\rightarrow\mathbb{R}$,  $$\psi_{x}(y ):= \psi(\dots, (y)_{-1}, (y)_{0}, \dot{(x)}_0, (x)_1, (x)_2, \dots)$$
and let $\chi_x:X^-\rightarrow\mathbb{R}$, 
 $$\chi_x := \exp\left(-\psi(\dots, (y)_{-1}, (y)_{0}, \dot{(x)}_0, (x)_1, (x)_2, \dots)\right)\cdot \mathbb{A}_{(y)_{0},(x)_0}.$$
 Clearly $\chi_x \in C_{c}^{+}(X^-)$, for every $x\in X^+$. \end{defn}
\begin{thm}
\label{thm:harm_conf_correspondence}
Let $(X^{},T)$ be a  transitive locally compact two-sided TMS and let  $\phi:X^+\rightarrow\mathbb{R}$ be a potential function with summable variations. Then, there is a 1-1 linear correspondence between the $\lambda$-conformal measures on $(X^-, T^-, \phi^-)$ and the eigenfunctions in $\mathcal{H}(\lambda)$ via the mapping $\pi$,  
$$(\pi(\mu))(x) = \mu(\chi_x). $$
\end{thm}
To prove the theorem, we will use to following elementary lemma. 
\begin{lem}
\label{lemma:cohomo_lemma}
Let $(X,T)$ be a transitive locally compact two-sided TMS and let $\phi^+ :X^+\rightarrow\mathbb{R}$, $\phi^-:X^-\rightarrow\mathbb{R}$  be  two potential functions with summable variations and which are cohomologous via a uniformly continuous transfer function $\psi:X\rightarrow\mathbb{R}$. Let $x\in X^+, y\in X^-$ and $a_1,\dots, a_n\in S$ s.t. $(y,a_n,\dots, a_1, \dot x)\in X$. Then, 
$$  e^{\phi^{+}_n(a_n\dots, a_1 x) - \psi(y\dot a_n,\dots, a_1 x)}=e^{\phi_n^-(ya_n,\dots, a_1) - \psi(y a_n,\dots, a_1 \dot x)}. $$
\end{lem}
\begin{proof}
By Equation (\ref{eq:cohomo_eq}),
\begin{align*}
&\phi^{+}_n(a_n\dots, a_1 x) - \phi_n^-(ya_n,\dots, a_1)\\ 
&=\sum_{i=1}^{n}\left( \phi^{+}(a_i,\dots, a_1x)- \phi^-(ya_n,\dots, a_i) \right)\\
&= \sum_{i=1}^n ( \psi(ya_{n},\dots, \dot a_i,\dots, a_1x) - (\psi\circ T)(ya_{n},\dots, \dot a_i,\dots, a_1x))\\
&= \psi(y\dot a_n,\dots, a_1 x) - \psi(ya_n,\dots, a_1\dot x).
\end{align*}
\end{proof}

\begin{proof}[Proof of Theorem \ref{thm:harm_conf_correspondence}]
Assume w.l.o.g. that $\lambda=1$.
 
\underline{$\pi(\mu)\in\harm$:} Let $h(x) = \mu(\chi_x)$.  
Since $\psi$ is uniformly continuous, for  $n\geq 2$ and  $x_1,x_2\in X^+$ with $d(x_1,x_2)\leq 2^{-n}$,
$$\chi_{x_1}(y)=e^{\pm o_n(1)} \chi_{x_2}(y), \quad \forall y\in X^-.$$In particular,$$h(x_1) = \mu(\chi_{x_1}) =e^{\pm o_n(1)}\mu(\chi_{x_2})=e^{\pm o_n(1)}h(x_2)$$
and $\log h$ is uniformly continuous.
Next, we write\begin{align*}
h(x) =& \mu(\chi_{x}) =(L^{*}_{\phi^-}(\mu))(\chi_{x}) \qquad(\because \mu \text{ is conformal})
\\ = &\sum_{b\in S}\int_{T^{-1}[b]}e^{\phi^-(yb)}\chi_{x}(yb)d\mu(y)\\\
=&  \sum_{b:x\in T[b]}\int_{T^{-1}[b]} e^{\phi^-(yb)-\psi(yb\dot x)}d\mu(y) \\
=& \sum_{b:x\in T[b]}\int_{T^{-1}[b]} e^{\phi(bx)-\psi(y\dot bx)}d\mu(y)\qquad(\because \text{ cohomology})\\ =& 
\sum_{b:x\in T[b]}\int_{X^{-}} e^{\phi(bx)}\chi_{bx}(y)d\mu(y)\\
=&\sum_{b:x\in T[b]} e^{\phi(bx)}h(bx)=L_{\phi}h(x).
\end{align*}
Hence $h(x)  =L_{\phi}h(x)$. 

\underline{$\pi$ is $1-1$:}
For a measure $\mu\in \conf(X^-)$, let $\mu_{}^x =\chi_x\mu$, i.e. $\mu^x(f) = \mu(\chi_x f)$ for all $f\in C_c(X^-)$.  Let $a_1,\dots, a_n\in S$ with $[a_n,\dots, a_1]\neq \varnothing$ and let $x\in T[a_1]$. 
Then, 
\begin{align*}
\mu^x[a_n,\dots, a_1] = & \int_{[a_n,\dots, a_1]} e^{-\psi^{}_x(y)}d\mu(y) \qquad \\
=& \int_{T^{-1}[a_n]}e^{\phi_n^-(ya_n,\dots, a_1) - \psi^{}(ya_n,\dots, a_1\dot x)}d\mu(y) \quad(\because \mu \text{ is conformal})\\
=& \int_{T^{-1}[a_n]}e^{\phi_n(a_n,\dots, a_1x) - \psi(y\dot a_n,\dots, a_1x)}d\mu(y) \quad(\because \text{ Lemma }\ref{lemma:cohomo_lemma})\\
=& e^{\phi^{}_n(a_n,\dots, a_1x)}\mu^{}(\chi_{a_n,\dots, a_1x}) .
\end{align*} 
Therefore, if $\pi(\mu_1)\equiv\ \pi(\mu_2)$, then $\mu_1^x=\mu_2^x$ for every $x$. This implies that
$$ \mu_1([a_n,\dots, a_1]) = \mu^x_1(e^{\psi^{}_x}1_{[a_n,\dots, a_1]}) =\mu^x_2(e^{\psi^{}_x}1_{[a_n,\dots, a_1]})=\mu_2([a_n,\dots, a_1]). $$

\underline{$\pi$ is onto:}
We first show that if   $L_\phi1=1$  then there is a conformal measure $\mu$ with $\pi(\mu)$=1. For this, we introduce a collection of finite measures $\{\mu^x\}_{x\in X^+} $ and construct a conformal measure $\mu$ with $\mu^x = \chi_x \mu$. Let 
$$ \mu^x ([a_n,\dots, a_1]) := L_\phi^n 1_{[a_n,\dots, a_1]}(x) = \begin{cases}e^{\phi_{n}(a_n,\dots, a_1x)} & x\in T[a_1] \\
0 & o.w. \\
\end{cases}.$$
Since $L_\phi1=1$, if $x\in T[a_1]$ then
\begin{align*}\sum_{b\in S }\mu^x ([ba_n,\dots, a_1]) =& \sum_b e^{\phi_{n+1}(ba_n,\dots, a_1x)}\\
 =&e^{\phi_{n}(a_n,\dots, a_1x)} \sum_b e^{\phi(ba_n,\dots, a_1x)} \\
 =& \mu^x[a_n,\dots, a_1]
\end{align*}
and $\mu^x$ can be extended to a probability measure on $X^-$ via Carath\'eodory''s extension theorem. 

Let $a\in S$. We show that if  $x_1,x_2\in T[a]$ then $e^{\psi_{x_1}}\mu^{x_1}=e^{\psi_{x_2}}\mu^{x_2}$ on $[a]$. Since every measurable set with finite measure can be approximated by a finite union of cylinders, it is suffices to consider cylinders of constant  length. 
Let $a_1,\dots, a_n$ with $a_1=a$ and $[a_n,\dots, a_1]\neq \varnothing$ and let $y_0\in T^{-1}[a_n]$. Then, 
\begin{align*}
& \int_{[a_n,\dots, a_1]} e^{\psi_{x_1}(y)}d\mu^{x_1}(y)\\
 &= e^{o_n(1)}e^{\psi_{x_1}(y_{0}a_n,\dots, a_1 )} \mu^{x_1}[a_n,\dots, a_1]\qquad(\because  \;\psi\text{ is uniformly continuous}) \\
&=e^{o_n(1)} e^{\psi(y_{0},a_n,\dots, a_1\dot x_{1})+\phi_n(a_n,\dots, a_1x_{1})}\\
&=e^{o_n(1)} e^{\psi(y_{0},\dot a_n,\dots, a_1 x_{1})+\phi^-_n(y_0, a_n\dots, a_1)}\qquad (\because \; \text{Lemma }\ref{lemma:cohomo_lemma})\\
&= e^{o_n(1)} e^{\psi(y_{0},\dot a_n,\dots, a_1 x_2)+\phi^-_n(y_0, a_n\dots, a_1)}\qquad(\because \;\psi\text{ is uniformly continuous}).
\end{align*}
Similarly, 
$$ \int_{[a_n,\dots, a_1]} e^{\psi_{x_2}(y)}d\mu^{x_2}(y) = e^{o_n(1)} e^{\psi(y_{0},\dot a_n,\dots, a_1 x_2)+\phi^-_n(y_0, a_n\dots, a_1)}. $$
Since $n$ can be arbitrarily large and $o_n(1)$ is uniform in $a_1,\dots, a_n$, $$e^{\psi_{x_1}}\mu^{x_1} = e^{\psi_{x_2}}\mu^{x_2}.$$ 

Let $\{b_i\}$ be an enumeration of $S$ and let 
\begin{equation}
\label{eq:onto_mu_construction}
\mu = \sum_{i}1_{[b_i]}e^{\psi_{x_{b_i}}}\mu^{x_{b_{i}}}
\end{equation} where $x_{b_i}\in T[b_i]$. Clearly $\mu$ is a positive Radon measure. We show that $\mu^x = \chi_x \mu$. Then, since $\mu^x$ are probability measures, 
$$\mu(\chi_x) = \mu^x(1) =1. $$
Fix $b\in S$, let $x\in [b]\subseteq X^{+}$ and let $f\in C_c(X^-)$.  Then, \begin{align*}
\mu(f\chi_x) = & \mu(fe^{-\psi_x}1_{T^{-1}[b]}) \\
=& \sum_{a\in S: \mathbb{A}_{a,b}=1} \mu(fe^{-\psi_x}1_{[a]}) \\
=& \sum_{a\in S:\mathbb{A}_{a,b}=1}\mu^{x_a}(fe^{\psi_{x_{a}}}e^{-\psi_x}1_{[a]})\\
=& \sum_{a\in S:\mathbb{A}_{a,b}=1}\mu^{x}(fe^{\psi_{x}}e^{-\psi_x}1_{[a]})\qquad (\because x,x_a\in T[a] )\\
=& \mu^x(f1_{T^{-1}[b]})=\mu^x(f).
\end{align*}

To show that $\mu$ is conformal, it is suffices to consider only cylinders. Let $[a_n,\dots, a_1]\neq \varnothing$, let $x\in T[a_1]$ and let $y_0\in T^{-1}[a_n] $. Then, 
\begin{align*}
(L_{\phi^-}^* \mu)[a_n,\dots, a_1] = & \int_{[a_n,\dots, a_2]} e^{\phi^-(ya_1)}d\mu(y) \\=& \int _{[a_n,\dots, a_2]}e^{\phi^-(ya_1)}e^{\psi_{x_{a_2}}(y)}d\mu^{x_{a_{2}}}(y)\qquad (\because \text{definition of }\mu )\\
=& \int _{[a_n,\dots, a_2]}e^{\phi^-(ya_1)}e^{\psi_{a_1x}(y)}d\mu^{{a_1}x}(y)\qquad(\because a_{1}x,x_{a_{2}}\in T[a_{2}] )\\
=& e^{\phi^-(y_0a_{n},\dots, a_1) +\psi(y_{0}a_{n},\dots, a_2 ,\dot a_1x)+\phi_{n-1}(a_n,\dots, a_1x)+o_{n}(1)}\\
=& e^{\phi_{n}(a_n,\dots, a_1x)+ \psi(y_{0}a_n,\dots, a_1\dot x)+o_n(1)}\qquad (\because \; \text{cohomology})\\
=&e^{o_n(1)} \int_{[a_n,\dots, a_1]} e^{ \psi_{x}(y)}d\mu^x(y)\\
=& e^{o_n(1)} \int_{[a_n,\dots, a_1]} e^{ \psi_{x_{a_{1}}}(y)}d\mu^{x_{a_1}}(y)\qquad (\because x,x_{a_{1}}\in T[a_{1}] )\\
=& e^{o_n(1)}\mu[a_n,\dots, a_1].
\end{align*}
  Again, since $n$ can be taken to be arbitrarily large and $o_n(1)$ is uniform in $a_1,\dots, a_n$, $\mu$ is indeed conformal w.r.t. $\phi^-$. 

Now, assume $h$ is an arbitrary positive eigenfunction with uniformly continuous logarithm and consider the potential  $\phi^h = \phi +\log h - \log h\circ T$. Then,
$$\phi^h - \phi^- = \psi +\log h - (\psi +\log h)\circ T. $$
Hence, the transfer function of $\phi^h$ and $\phi^-$ is $\psi^h = \psi + \log h$. Observe that in the construction of $\mu$ in Eq. (\ref{eq:onto_mu_construction}), we only assumed that $\phi$ and $\psi$ are uniformly continuous.\ Since $\log h$ is uniformly continuous, $\phi^h$ and $\psi^h$ are uniformly continuous as well. Since $L_{\phi^h}1=1$, there exists a measure $\mu$ which is conformal  w.r.t. $\phi^-$  and
$$\mu(e^{-\psi^h_x}1_{T^{-1}[(x)_0]})=1. $$
Since $$e^{-\psi^h_x}1_{T^{-1}[(x)_0]} = e^{-\psi_x}1_{T^{-1}[(x)_0]}h(x)=\chi_x h(x)$$
then 
$$h(x) = \mu(\chi_x). $$
\end{proof}
\begin{remk}
We emphasize that the correspondence established in Theorem \ref{thm:harm_conf_correspondence} is valid also for the recurrent potentials ($\sum_{n\geq 0}\lambda^{-n} L_\phi^n 1_{[a]}=\infty, \lambda = \exp P_G(\phi)$) where the $\lambda$-eigenfunction and the $\lambda$-eigenmeasure are both unique up to a constant by \cite{sarig_2001}. 
\end{remk}
Since the reversed Martin boundary is not empty, Theorem \ref{thm:harm_conf_correspondence} implies directly the existence of positive eigenfunctions.
\begin{thm}
Let $(X^{},T)$ be a  transitive locally compact two-sided TMS and let  $\phi:X^{+}\rightarrow\mathbb{R}$ be a  $\lambda$-transient potential function with summable variations.  Then, there exists a positive continuous $\lambda$-eigenfunction.
\end{thm}
With the correspondence of Theorem \ref{thm:harm_conf_correspondence}, one can easily obtain analogues to the results of Section \ref{section:integration}: 
\begin{defn}
A function $h\in \harm(\lambda)$ is \textit{$\lambda$-minimal } if for every $h'\in \harm(\lambda)$ with $h'\leq h$, we have that $h' \propto\ h$. 
\end{defn}
\begin{thm}
\label{thm:harmonic_rep}
Let $(X^{},T)$ be a  transitive locally compact two-sided TMS and let  $\phi:X^{+}\rightarrow\mathbb{R}$ be a  $\lambda$-transient potential function with summable variations. Then, for every $h\in \harm(\lambda)$, there exists a unique finite measure $\nu$ on $\cev{\mathcal{M}}_m(\lambda)$ s.t.
$$h(x) = \int_{\cev{\mathcal{M}}_m(\lambda)} \cev{K}(\chi_x,\xi|\lambda)d\nu(\xi) \quad \forall x\in X^+. $$
Moreover, $h$ is $\lambda$-minimal iff $\nu$ is a dirac measure, meaning $\pi^{-1} h \in \cev{\mathcal{M}}_m(\lambda)$.
\end{thm}
\begin{proof}
The existence follows directly from Theorem \ref{thm:minimal_kernel_rep} and Theorem \ref{thm:harm_conf_correspondence}. Since $\pi$ is linear, so is $\pi^{-1}$ and thus $h$ is minimal iff $\pi^{-1}h$ is extremal. As for the uniqueness, if there exist $\nu$ and $\nu'$ s.t. 
$$h(x) = \int_{\cev{\mathcal{M}}_m(\lambda)} \cev{K}(\chi_x,\xi|\lambda)d\nu(\xi)= \int_{\cev{\mathcal{M}}_m(\lambda)} \cev{K}(\chi_x,\xi|\lambda)d\nu'(\xi) \quad \forall x\in X^+  $$
then with $\mu = \int K(\cdot, \xi|\lambda) d\nu(\xi)$ and $\mu' = \int K(\cdot, \xi|\lambda) d\nu'(\xi)$ we have that $\pi\mu = \pi\mu'=h$. Since $\pi$ is $1-1$, $\mu=\mu'$. By Theorem \ref{thm:minimal_kernel_rep} we must have that $\nu=\nu'$.    
\end{proof}

In several applications, we consider $T$-invariant measures on $X$ of the form $m=h\mu$, where $h$ is a positive eigenfunction and $\mu$ is a positive eigenmeasure. The main results of  Theorem \ref{thm:boundary_rep_theorem} and Theorem \ref{thm:harmonic_rep} lead to a different interpretation for such a construction: \begin{enumerate}
\item 
Pick a finite measure $\nu^+$ on $\mathcal{M}_m$  and set $\mu^+= \int_{\mathcal{M}_m}\mu_\omega d\nu^+(\omega). $ 
\item 
Pick a finite measure $\nu^-$ on $\cev{\mathcal{M}}_m$  and set $\mu^-= \int_{\cev{\mathcal{M}}_m}\mu_\omega d\nu^-(\omega). $ 
\item
The resulting $T$-invariant  measure $m$ is, with $f\in C_c(X)$, 
\begin{equation}
\label{eq:hmu_rep}
m(f) = \int_{x\in X^+}\int_{y\in X^-} f(\dots, (y)_{-1}, (y)_0, (x)_0, (x)_1,\dots)  \chi_x(y)d\mu^-(y)d\mu^+(x). 
\end{equation}
\end{enumerate}
Theorem \ref{thm:harm_conf_correspondence} is valid in the recurrent case as well and so does Eq. (\ref{eq:hmu_rep}), although the measures $\mu^-$ and $\mu^+$ are unique up to normalization. See \cite{kaimanovich_1990,kaimanovich_1994} for similar decompositions in different settings.

To conclude the discussion on the reversed Martin boundary, we provide an example which shows that $\mathcal{M}$ and $\cev{\mathcal{M}}$ can be different.

\noindent\textbf{Example 1.}
\label{example:different_bounderies}
Consider $S=\mathbb{Z} \cup \{n':n\in \mathbb{N}\}$, where $n'$ is a different copy of $n$ and consider the   transition matrix $\mathbb{A}$ with $\mathbb{A}_{a,b}=1$ iff  one of the following cases 
\begin{itemize}
\item 
$a=0, b\in \{-1,+1\}. $
\item
$a=1', b=0. $
\item
$a,b\in \mathbb{Z}\setminus \{0\}\text{ and } b=a + sign(a).  $  
\item
$a\in \mathbb{Z}\setminus\{0\},b=n' \text{ with }|a|=n.  $
\item
$ a=(n+1)', b=n'$.
\end{itemize}

See Figure \ref{figure:different_bounderies}. Clearly the corresponding TMS is locally finite and transitive.
\begin{prop}
There is  $\alpha<0$ s.t. the potential function $\phi\equiv \alpha$ is transient, $\mathcal{M}$ contains  at least two points and $\cev{\mathcal{M}}$ contains a single point alone. 
\end{prop}
\begin{proof}
Since the out-degree of any state is bounded by $2$,  for any $x\in X^+$ 
$$(L_\phi^n 1_{[o]})(x) \leq e^{n\alpha} 2^n $$
whence, with $\alpha < -\log2$, $G(1_{[o]},x)<\infty$ and the potential $\phi \equiv \alpha$ is indeed transient.

 In the reversed graph, $n'$ with $n\rightarrow\infty$ is the only possible direction which escapes every finite set. Hence\ $\cev{\mathcal{M}}$ cannot contain more than a single point.  In the original graph, $n\rightarrow \infty$ or $n\rightarrow -\infty$ are the only two   possible directions to escape every finite set. We show that they may correspond to two different points in $\mathcal{M}$. Let $x_n^+\in [n]$ and let $x_n^- \in [-n]$. Let
$$Z_n^*(\alpha,  a,b) \\
= e^{n\alpha}\cdot \#\left\{\begin{array}{c}
\text{paths of length }n \text{ from }a\text{ to }b,  \\
\text{ first reaching }b \text{ in the }n^{\text{th}} \text{ step} \\
\end{array}\right\}$$
and let 
$$ F(\alpha,a,b) = \sum_{n\geq 0}Z_n^* (\alpha, a,b).$$
Observe that 
$$L_\phi^n 1_{[a]}(x) = e^{n\alpha}\cdot \#\{\text{paths from }a\text{ to }(x)_0 \text{ of length }n\}. $$
Since every path from $-1$ to $n$ must pass through $1$, $$G(1_{[-1]},x_{n}^{+}) = F(\alpha,-1,1) G(1_{[1]},x_{n}^{+}). $$
By the symmetry of the graph, $G(1_{[1]},x_{n}^+) = G(1_{[-1]},x_{n^{}}^-)$. Hence, 
$$K(1_{[-1]},x_{n}^{+}) = F(\alpha,-1,1)K(1_{[-1]},x_{n}^-).$$
Since $F(\alpha, -1,1)$ varies with $\alpha$,  we can decrease $\alpha$ so that $F(\alpha, -1, 1)\neq1$. Observe that decreasing $\alpha$ does not affect the transience of $\phi$. Then, 
\begin{align*}
K(1_{[-1]},\infty) =& \lim_{n\rightarrow\infty} K(1_{[-1]},x_{n}^{+})\\
=&\lim_{n\rightarrow\infty}F(\alpha, -1, 1)K(1_{[-1]},x_{n}^{-})\\
=& F(\alpha, -1,1)K(1_{[-1]},-\infty)  
\end{align*}
and in particular $K(\cdot,\infty)\neq K(\cdot, -\infty)$. 
\end{proof}

        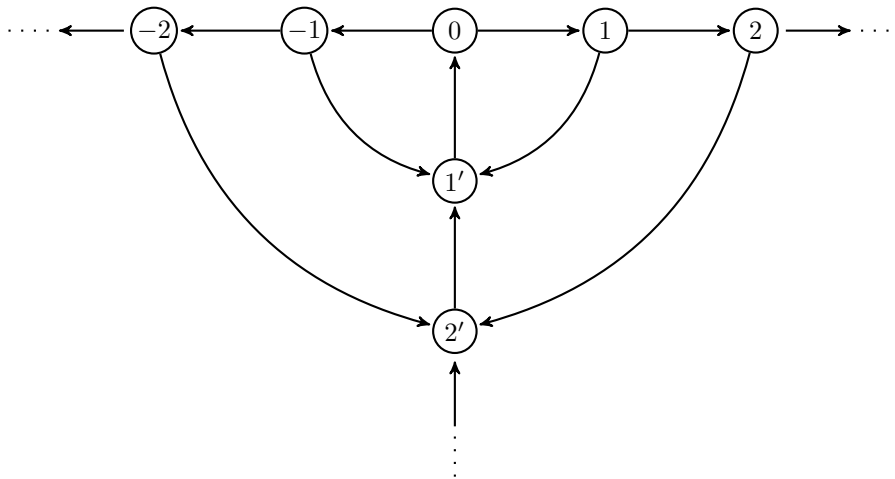
\begin{figure}[!htb]

\begin{center}
\begin{tikzpicture}[-,>=stealth',shorten >=1pt,auto,
                    thick,main node/.style={circle,draw,font=\sffamily\Large\bfseries}]

  \node[circle,draw,inner sep=3pt, minimum size=0pt] at (0,0) (0) {$0$};
  \node[circle,draw,inner sep=3pt, minimum size=0pt] at (2,0) (1) {$1$};
  \node[circle,draw,inner sep=3pt, minimum size=0pt] at (4,0) (2) {$2$};
  \node[circle,draw,inner sep=1pt, minimum size=0pt] at (-2,0) (-1) {$-1$};
  \node[circle,draw,inner sep=1pt, minimum size=0pt] at (-4,0) (-2) {$-2$}; 
  \node[circle,draw,inner sep=2pt, minimum size=0pt] at (0,-2) (1') {$1'$};
  \node[circle,draw,inner sep=2pt, minimum size=0pt] at (0,-4) (2') {$2'$};

  \path[every node/.style={font=\sffamily\small}]
    (0) edge [->]  (1)    
    (1) edge [->]  (2)    
    (0) edge [->]  (-1)    
    (-1) edge [->]  (-2)    
    (1') edge [->]  (0)    
    (2') edge [->]  (1') 
      
    (-1) edge [bend right=30,->]  (1')    
    (-2) edge[bend right=30,->]  (2')
    (1) edge [bend right=-30,->]  (1')    
    (2) edge[bend right=-30,->]  (2')
;

\draw [->, thick] (-4.4,0) -- (-5.3,0);
\draw [loosely dotted, thick] (-5.4,0) -- (-6.1,0);
\draw [->, thick] (4.4,0) -- (5.3,0);
\draw [loosely dotted, thick] (5.4,0) -- (6.1,0);
\draw [<-, thick] (0,-4.4) -- (0,-5.3);
\draw [loosely dotted, thick] (0,-5.4) -- (0,-6.1);

\end{tikzpicture}
\end{center}
\caption{The states graph in Example 1.}
\label{figure:different_bounderies}
\end{figure}

\section{Poisson boundary and dominated eigenfunctions}
The purpose of the probabilistic Poisson boundary   is to describe all bounded harmonic functions, equivalently all positive eigenfunctions which are bounded by the eigenfunction $h \equiv 1$;
  see Section \ref{section:random_walk} for the connection between the eigenfunctions and the harmonic functions. For more on the probabilistic Poisson boundary, see \cite{kaimanovich_1995}.

  In the non-probabilistic setup, with general $\phi$, the constant function $h \equiv 1$ may not be  an eigenfunction of the   Ruelle operator. In particular,  there is no canonical choice for the dominating eigenfunction.
Here, we generalize the notion of the Poisson boundary to study all eigenfunctions which are dominated by some fixed eigenfunction $h$.
In particular, we derive an integral representation of dominated eigenfunctions on the boundary and prove an almost-sure convergence to some ratio.  

In contrast to the previous parts of this paper, the discussion in this section has a measure-theoretic flavour. In particular, the  Poisson boundary is studied as a measure space unlike the Martin boundary which is studied as a topological space. \begin{defn}
Let $h\in \mathcal{H}^+(\lambda)$, let $\mu_h =\pi^{-1}(h)$ be the corresponding measure on $X^-$ and let $\nu_h$ be  the unique measure on $\cev{\mathcal{M}}_m(\lambda)$ s.t. for all $x\in X^+$
$$h (x) = \int_{\cev{\mathcal{M}}_m(\lambda) } \cev{K}(\chi_x, \omega|\lambda)d\nu_h(\omega). $$
See Theorems \ref{thm:harm_conf_correspondence} and \ref{thm:harmonic_rep}. Let
$$\mathcal{H}^\infty (\lambda,h)\\
=\left\{f\in C(X^{+}) : \begin{array}{c}
\exists c>0 \text{ s.t. } |f|\leq c\cdot h,\; L_\phi f=\lambda f \text{ and } \\
\log f\text{ is uniformly continuous} \\
\end{array} \right\}.
$$
The \textit{($\lambda, h)$-Poisson boundary }is the measure space $(\cev{\mathcal{M}}_m(\lambda), \nu_h)$.
\end{defn}
Clearly every function $\varphi \in L^{\infty}(\cev{\mathcal{M}}_m(\lambda), \nu_h)$ defines a function in $f_{\varphi}\in\mathcal{H}^\infty (\lambda,h)$ by
$$ f_{\varphi}(x) = \int \cev{K}(\chi_x, \omega|\lambda)\varphi(\omega)d\nu_h(\omega).$$ We show that every $f \in \mathcal{H}^\infty(\lambda, h)$ is uniquely determined by a function $\varphi_f\in L^\infty(\cev{\mathcal{M}}_m(\lambda), \nu_h)$ which is the limit of $\frac{f}{h}$  in  $\cev{\mathcal{M}}_m$.
To prove these results, we introduce the boundary map and show that it is Borel-measurable.\begin{defn} Given a Poisson boundary $(\cev{\mathcal{M}}_m(\lambda), \nu_h)$, we define the \textit{boundary map} $\textbf{Bnd}: (X^{-}, \mu_h) \rightarrow (\cev{\mathcal{M}}_m(\lambda), \nu_h)$, 
$$\textbf{Bnd}(y) = \lim_{n\rightarrow\infty}T^{-n}y_. $$
By Theorem \ref{thm:conv_to_boundary}, the limit $\mu_h$-a.s. exists and $\textbf{Bnd}$ is well-defined.
\end{defn}\begin{prop}
Let $(X^{-},T^{-1})$ be a transitive locally compact negative one-sided TMS and let $\phi$ be a $\lambda$-transient potential function with summable variations.  Then, the map $\textbf{Bnd}$ is Borel-measurable.
\end{prop}
\begin{proof}
Let $A \subseteq \cev{\mathcal{M}}_m$ be a $\rho$-Borel set. By Proposition \ref{claim:compactification_prop}, for every $\epsilon>0$, the set
$$B_\rho(A,\epsilon) := \left\{y\in X^{-} : \rho(A,y) < \epsilon\right\}$$ 
is $d$-open. Since we can write
$$\omega^{-1}(A) = \bigcap_{m\geq 1}  \bigcup_{n\geq 0}\bigcap_{k\geq n}T^{-k} B_\rho(A,1/m)$$
the map is indeed $d$-measurable.
\end{proof}

\begin{thm}
Let $(X,T)$ be a transitive locally compact two-sided TMS and let $\phi$ be a $\lambda$-transient potential function with summable variations. 
Let $h\in \mathcal{H}^+(\lambda)$ and let $f\in \mathcal{H}^\infty(\lambda,h)$. Then, there exists a unique $\varphi_f \in L^\infty(\cev{\mathcal{M}}_m, \nu_h)$ s.t. for all $x\in X^+$,
$$ f (x) = \int \cev{K}(\chi_x, \omega|\lambda) \varphi_{f}(\omega) d\nu_h(\omega)$$
 and for $\mu_h$-a.e. $y\in T^{-1}[(x)_0] \subseteq  X^-$,
$$\frac{f(y_{-n},\dots, y_0 x)}{h(y_{-n},\dots, y_0 x)} \xrightarrow[n\rightarrow\infty]{} \varphi_f(\textbf{Bnd}(y)). $$
\end{thm}
\begin{proof}
Assume w.l.o.g. that $\lambda=1$. Assume first that $0\leq f \leq h$. Then, $f\in \mathcal{H}^+$ and by Theorem \ref{thm:harmonic_rep} there is measure $\nu_h$ on $\cev{\mathcal{M}}_m$ s.t.
$$f(x) = \int \cev{K}(\chi_x, \omega) d\nu_f(\omega).$$
Since the bijection $\pi$ from Theorem \ref{thm:harm_conf_correspondence} is linear and $f \leq h$, we have that $\nu_f \leq \nu_h$  and the Radon-Nikodym $\varphi_f = \frac{d\nu_f}{d\nu_h}$ exists and  is bounded by $1$. In particular, 
$$f(x) = \int \cev{K}(\chi_x, \omega)\varphi_f(\omega)  d\nu_h(\omega). $$
For a general $f\in \mathcal{H}^\infty(\lambda, h)$ with $|f|\leq c\cdot h$, consider $f_{1} =  \frac{1}{2}(h + c^{-1}f)$. Clearly $0~\leq f_{1} \leq h$ and so there is $\varphi_{f_1}\in L^1(\cev{\mathcal{M}}_m, \nu_h)$ with
$$f_1(x) = \int \cev{K}(\chi_x, \omega)\varphi_{f_1}(\omega)  d\nu_h(\omega).  $$
Then, with $\varphi_{f} = 2c \varphi_{f_1} - c$, 
\begin{align*} \int \cev{K}(\chi_x, \omega)\varphi_{f}(\omega)  d\nu_h(\omega) =&2c \int \cev{K}(\chi_x, \omega)\varphi_{f_1}(\omega)  d\nu_h(\omega) -c  \int \cev{K}(\chi_x, \omega)  d\nu_h(\omega)\\
=&2c f_1(x) - ch(x) = f(x).
  \end{align*}

Let $\mu_h^x = \chi_x \mu_h$. Recall that $\mu_h^x$ is a finite measure on $X^-$, see proof of Theorem \ref{thm:harm_conf_correspondence}. Let $\tilde \varphi_f\in L^1(X^-, \mu^x)$, 
$$\tilde \varphi_f(y) = \varphi_f(\textbf{Bnd}(y)).  $$
 \begin{lem}
 For all $\omega\in \cev{\mathcal{M}}_m$ and all $g\in C_c(\chi_x)$, 
 $$\cev{K}(g \tilde \varphi_f, \omega) = \cev{K}(g, \omega) \varphi_f(\omega).  \ $$
 \end{lem}
 \begin{proof}Fix $\omega \in \cev{\mathcal{M}}_m$. Since $\textbf{Bnd}$ is a tail-invariant function, for all $y\in X^-$, 
 $$\cev{K}(g \tilde {\varphi_f}, y) = \cev{K}(g, y)\tilde  \varphi_f(z). $$
Let $y\in X^-$ be a typical point w.r.t. the measure $\mu_\omega$, namely $T^{-n}y \rightarrow \omega$ in $\widehat{X^-}$. Then, for all $n\geq 0$, 
 $$\cev{K}(g \tilde \varphi_f,T^{-n} y) = \cev{K}(g, T^{-n} y)\tilde  \varphi_f(T^{-n} y) =\cev{K}(g, T^{-n} y)  \varphi_f(\omega) .$$
The lemma follows  by taking $n\rightarrow\infty$. 
 \end{proof}
 By  the auxiliary lemma,
\begin{align*}\mu_{h}^x(\tilde {\varphi_f}) =&\mu_{h}( \chi_x\cdot \tilde {\varphi_f})\\
=& \int_{\cev{\mathcal{M}}_m} \cev{K}(\chi_x\cdot  \tilde {\varphi_f},\omega)d\nu_h(\omega) \\
=& \int_{\cev{\mathcal{M}}_m}  \cev{K}(\chi_x,\omega) \varphi_{f}(\omega)d\nu_h(\omega)=f(x).
\end{align*}
Let $\mathcal{B}_n$ be the $\sigma$-algebra generated by the cylinders of length $n$ in $X^-$ and let
$$g_n(y) = \mathbb{E}_{\mu^x_h}[\tilde {\varphi_f}|\mathcal{B}_n](y) $$
the conditional expectation of $\tilde {\varphi_f}$ w.r.t. the $\sigma$-algebra $\mathcal{B}_n$. We show that
$$g_n(y) =\frac{f(y_{-n+1},\dots, y_0 x)}{h(y_{-n+1},\dots, y_0 x)}  $$
Then, by the martingale convergence theorem,
$$\lim_{n\rightarrow\infty}\frac{f(y_{-n},\dots, y_0 x)}{h(y_{-n},\dots, y_0 x)}=\lim_{n\rightarrow\infty}\mathbb{E}_{\mu^x_h}[\tilde {\varphi_f}|\mathcal{B}_n](y)  = \tilde {\varphi_f}(y). $$

The conditional expectation on the $\sigma$-algebra $\mathcal{B}_n$ has the following formula 
$$g_{n}(y) = \frac{1}{\mu_{h}^x[(y)_{-n+1},\dots, (y)_0]} \mu_{h}^x(\tilde {\varphi_f} \cdot 1_{[(y)_{-n+1},\dots, (y)_0]}).$$
Let $g:X^-\rightarrow X^-$ be a $T^{-1}$-invariant function. We derive, with $a_n,\dots, a_1\in S$ admissible, \begin{align*}
&\mu^x_{h}(g\cdot 1_{[a_n,\dots, a_1]}) \\
&=  \mu^x_{h}(g\circ T^{-n}\cdot 1_{[a_n,\dots, a_1]}) \\
&= \int_{[a_n,\dots, a_1]} g( T^{-n}y)e^{-\psi^{}_x(y)}d\mu_{h}(y) \qquad \\
&= \int_{T^{-1}[a_n]}g(y)e^{\phi_n^-(ya_n,\dots, a_1) - \psi^{}(ya_n,\dots, a_1\dot x)}d\mu_{h}(y) \qquad(\because \mu_{h} \text{ is } L_{\phi^-}\text{-conformal})\\
&= \int_{T^{-1}[a_n]}g(y)e^{\phi_n(a_n,\dots, a_1x) - \psi(y\dot a_n,\dots, a_1x)}d\mu_{h}(y) \qquad(\because \text{ Lemma }\ref{lemma:cohomo_lemma})\\
&= e^{\phi^{}_n(a_n,\dots, a_1x)}\mu_{h}^{}(g\chi_{a_n,\dots, a_1x}) .
\end{align*} 
Since $\mu_h(\chi_{a_n,\dots, a_1x}) = h(a_n,\dots, a_1 x)$, 
$$g_{n}(y)=\frac{\mu_{h}^x(\tilde {\varphi_f} \cdot 1_{[(y)_{-n+1},\dots, (y)_0]})}{\mu_{h}^x[(y)_{-n+1},\dots, (y)_0]} =\frac{f(y_{-n+1},\dots, y_0 x)}{h(y_{-n+1},\dots, y_0 x)}. $$

\end{proof}

\section{Applications to first-order phase transitions}
\subsection{Background}
In this section we apply our results to the theory of Gibbs states and first order phase transitions.
Recall that a thermodynamic system is said to undergo a {\em phase transition of the first order} if there are several possible equilibrium values to some of its thermodynamic quantities. The question of how to formalize this was studied extensively in the sixties, see e.g. \cite{ruelle_1969,lanford_1969,dobrushin_1968}. Here we follow the program of Dobrushin, Lanford and Ruelle which formalizes a phase transition of the first order as a situation where there are several Dobrushin-Lanford-Ruelle (DLR) measures, see Section \ref{Section-DLR} below.
An alternative approach to first-order phase transitions is to view them as situations where the thermodynamic limit is not unique, see Section \ref{sec:TDL} below. The two approaches are often equivalent, see \cite{ruelle_2004}.

We show here that if $\phi$ is transient with a Martin Boundary bigger than one point, then there are several different non-singular DLR states, each of which corresponds to thermodynamic limits where the ``boundary conditions" escape to infinity in different directions. Compare with \cite{follmer_1975}.

\subsection{Existence and non-uniqueness of DLR measures}\label{Section-DLR} Recall that  $(X^+,T)$ is a topologically mixing locally compact countable Markov shift and that $\phi:X^+\to\mathbb R$  is a potential function with summable variations and finite Gurevich pressure. The following definition is a version of the classical definition of a DLR measure, tailored to fit our one-dimensional, one-sided, infinite state scenario. See \cite{dobrushin_1968,lanford_1969,georgii_2011} for more general cases.

\begin{defn}We say that a probability measure $m$ is a {\em Dobrushin-Lanford-Ruelle measure} for $\phi$ if
for all $n\geq 1$ and $m$-a.e. $x\in X^+$,\begin{equation}
\label{eq:dlr_condition}
\mathbb{E}_m[1_{[(x)_0,,\dots, (x)_{n-1}]}|T^{-n}\mathscr{B}](x) = \frac{e^{\phi_n(x)}}{\sum_{y:T^ny=T^nx}e^{\phi_n(y)}}
 \end{equation}
 where  $\mathscr{B}$ is the Borel $\sigma$-algebra of $X^+$ and $\mathbb{E}_m[\cdot | T^{-n}\mathscr{B}]$ is the conditional expectation of $m$ w.r.t. the $\sigma$-algebra $T^{-n}\mathscr{B}$.
\end{defn}

 Recall that a positive Radon measure $\mu$ is \textit{non-singular} if for every Borel set $A\subseteq X^+$, $\mu(A)=0$ iff $\mu(T^{-1}A)=0$. The connection between DLR measures and eigenmeasures is explained in the following  propositions.  \begin{prop}[]\label{Prop-ConformalIsDLR} Let $\phi:X^{+}\rightarrow\mathbb{R}$ be a Borel function and let $\nu$ be a non-singular  probability measure  with $L_\phi^\ast\nu=\lambda\nu$ for some $\lambda>0$.
Then $\nu$ is a non-singular DLR measure for $\phi$. 
\end{prop}
\begin{proof}
See \cite{petersen_1997,sarig_2015}. 
\end{proof}

\begin{prop}\label{Prop-DLRisConformal} 
 Let $\nu$ be a non-singular DLR measure for $\phi$. Then, there exists a function $h:X^+\rightarrow\mathbb{R}$, which is measurable w.r.t. the $\sigma$-algebra $\cap_{n\geq0}T^{-n}\mathscr{B}$, s.t. $L_{\phi + h}^\ast\nu=\nu$.

\end{prop}
\begin{proof}
See appendix. 
\end{proof}
In the infinite state case, there may exist DLR measures which are not non-singular. These measures may not correspond to eigenmeasures of $L_\phi$, see Example 2 in the appendix.

 We would like to relate the richness of the Martin boundary to first-order phase-transitions. However, the resulting conformal measures in Section \ref{section:construction} may be infinite. To overcome this problem, one can ``adjust" the conformal measures by a uniformly continuous density to obtain conformal probability measures  w.r.t. a different but cohomologous potential function.

\begin{prop}\label{cor:transient_DLR}
Assume that $X^+$ is locally compact and topologically mixing and that $\phi:X^+\rightarrow\mathbb{R}$ is $\lambda$-transient and has summable variations. Then, there exists a uniformly continuous function  $h:X^+\rightarrow\mathbb{R}$ s.t. for every $\mu\in \mathcal{M}(\lambda)$, the measure $\frac{1}{\mu(h)}h\mu$ is a DLR  measure w.r.t. $\phi -\log h + \log h\circ T$. 
\end{prop}
\begin{proof}
Let $C_a>0$ the constant from Lemma \ref{lemma:kernel_boundness} and consider
the function$$h(x) = \sum_{n=1}^\infty\frac{1}{2^n C_{a_n}}1_{[a_n]} $$
where $\{a_n\}$ is an enumeration of $S$. Then, for every $\mu\in \mathcal{M}(\lambda)$, the measure $d\mu^h = hd\mu $ is finite and $L_{\phi -\log h + \log h\circ T}^*\mu^h = \lambda \mu^h$. Therefore, by Proposition \ref{Prop-ConformalIsDLR}, $\frac{1}{\mu^h(1)}\mu^h$ is a DLR measure.
\end{proof}
Proposition \ref{cor:transient_DLR} implies that if the Martin Boundary of $\phi-\log h+\log h\circ T$ contains more than one point, then $\phi-\log h+\log h\circ T$ has more than one non-singular DLR state, namely a first order phase transition. Examples include simple random walks on $d$-regular trees ($d\geq 3$) \cite{sawyer_1997} and Example 1 in Section \ref{section:harmonic_functions}.

\subsection{Thermodynamic limits}
\label{sec:TDL}
 We will now interpret the  DLR states arising from different points in the Martin boundary as  thermodynamic limits with different boundary conditions. 

In our context, thermodynamic limits arise from the following scheme:
\begin{enumerate}
\item Approximate $X^{+}$ with finite  subsets $X_N$ by imposing a boundary condition which rules out all but a discrete collection of  configurations.
\item Define the ``canonical ensemble" $\mu_N$ on $X_N$ by giving configurations weights according to the Gibbs formula and then normalizing as possible.
\item Pass to the limit in some regime where $X_N$ fills $X^{+}$ densely. The weak star limit points of $\mu_N$ are called thermodynamic limits or Gibbs states.
\end{enumerate}
The mathematical question is which limiting regimes give weak-star convergence, and what are the limiting measures.

We describe here the limiting regimes which work in the positive recurrent and the null recurrent scenarios. For more on positive and null recurrence, see \cite{sarig_2001}. To simplify calculations, we assume that $P_G(\phi)=0$ and that $T:X^+\rightarrow X^+$ is topologically mixing.  

\textbf{The positive recurrent case:} $L_\phi^n1_{[o]}$ is eventually bounded below \cite{sarig_2001}.\begin{enumerate}
\item Fix $x\in X^+$ and let
$$X_N(x):=\{y\in X^+: y_N^\infty=x_0^\infty\}\equiv \{y: T^N y=x\}.$$
\item 
Set
$$\mu_N=\frac{1}{Z_N(x)}\sum_{y\in X_N(x)} e^{\phi_N(y)}\delta_y$$ where $Z_N(x)=\sum_{y\in X_N(x)} e^{\phi_N(y)}1_{[o]}(y)$. By the generalized Ruelle's Perron-Frobenius theorem \cite{sarig_1999,sarig_2001}  for every $f\in C_c^+(X^+)$
$$
\mu_N(f)=\frac{(L_\phi^N f)(x)}{(L_\phi^N 1_{[o]}(y))(x)}\xrightarrow[N\to\infty]{}\frac{\mu(f)}{\mu[o]}
$$
\end{enumerate}
where $\mu$ is the unique eigemeasure (which is also a DLR state).

\textbf{The null recurrent case:} $L_\phi^n1_{[o]}\xrightarrow[n\rightarrow\infty]{}0$ but $\sum_{n}L_\phi^n1_{[o]}=\infty$. Now, the previous procedure is problematic because the numerator and denominator both tend to zero. 
So we use the following alternative scheme.
\begin{enumerate}
\item Fix $x\in X^+$ and let 
$$X_N'(x):=\{y\in X^+: y_n^\infty=x_0^\infty\text{ for }0\leq n< N\}=\bigcup_{n=0}^{N-1} X_n(x).$$
\item 
Set
$$\mu_N=\frac{1}{Z_N'(x)}{\sum_{n=0}^{N-1}\sum_{y\in X_n'(x)} e^{\phi_n(y)}\delta_y}.$$ where $$Z_N'(x) ={\sum_{n=0}^{N-1}\sum_{y\in X_n'(x)} e^{\phi_n(y)}1_{[o]}}(y).$$
Again, by the generalized Ruelle's Perron-Frobenius theorem \cite{sarig_2001}
$$
\mu_N(f)=\frac{\sum_{n=0}^{N-1}(L_\phi^n f)(x)}{\sum_{n=0}^{N-1} (L_\phi^n 1_{[o]})(x)}\xrightarrow[N\to\infty]{}\frac{\mu(f)}{\mu[o]} ,
$$
\end{enumerate}
where $\mu$ is the unique eigenmeasure (which is also a DLR state for $\phi-\log h + \log h\circ T$, see  Proposition \ref{cor:transient_DLR}).

\textbf{The transient case:}  $\sum_{n}L_\phi^n 1_{[o]}<\infty$ . Now  the ``edge effects" do not vanish in the limit and a different procedure is required. We suggest here the following limiting regime which avoids this issue. Let $\mathcal M=\mathcal M(\exp P_G(\phi))$ the Martin boundary of $\phi$.
\begin{enumerate}
\item Fix $\omega\in \mathcal M$ and $x\in X^+$ s.t. $T^n x\xrightarrow[n\to\infty]{}\omega$. By Theorem \ref{thm:conv_to_boundary}, such $x$ exists. Take
$$
X_N''(x):=\{y\in X^+: \exists m>0\text{ s.t. }y_m^\infty=x_N^\infty\}=\bigcup_{m=0}^\infty X_m(T^N x).
$$
Notice that $\bigcup_{N}X''_N(x)=\{y: \exists m, d(T^{n+m}y,T^nx)\rightarrow 0\}$ is the symbolic analogue of the weak-star manifold of $x$.  
\item Set 
$$\mu_N=\frac{\sum_{m=0}^\infty\sum_{y\in X_m(T^N x)} e^{\phi_m(y)}\delta_y}{\sum_{m=0}^\infty\sum_{y\in X_m(T^N x)} e^{\phi_m(y)}1_{[o]}(y)}.$$ By Proposition \ref{claim:compactification_prop},
$$
\mu_N(f)=\frac{G(f,T^nx)}{G(1_{[o]},T^Nx)}=K(f,T^Nx)\xrightarrow[N\to\infty]{}\mu_\omega(f)
$$
\end{enumerate}
where $\mu_\omega$ is given in Definition \ref{def:omega_measure}.
Again, since $\mu_\omega$ is an eigenmeasure,  the thermodynamic limit is a DLR state for $\phi-\log h + \log h\circ T$, see  Proposition \ref{cor:transient_DLR}. However, this time the choice of boundary condition $x$ matters; if we work with a different boundary condition $x'$ with $T^N x'\xrightarrow[N\to\infty]{}\omega'\neq \omega$, then the thermodynamic limit we will get is $\mu_{\omega'}\neq \mu_\omega$.

\section{The Martin boundary of a  transient random walk}
\label{section:random_walk}
In this section we illustrate why the boundaries constructed in Section  \ref{section:construction} and  Theorems \ref{thm:minimal_kernel_rep} and \ref{thm:conv_to_boundary} are, in some sense, a generalization of the  probabilistic Martin boundary.  Let $(Z,P)$ be a random walk on a countable set $S$, with random variable $Z=(Z_n) \in X^{+}$ and $P:S\times S\rightarrow[0,1]$ a probability transition matrix.  Recall that a function $h:S\rightarrow\mathbb{R}$ is  $P$-\textit{harmonic} if
$$h(a) = \sum_{b}P(a,b)h(b), \quad \forall a\in S. $$ 
\begin{thm}
\label{thm:prob_martin_boundary}
Assume that the walk is transient, locally finite and irreducible. 

\begin{enumerate}

\item 
\label{thm:prob_martin_boundary1}
Let $h:S\rightarrow\mathbb{R}^+$ be a positive $P$-harmonic function. Then, there exists a unique  measure $\nu$ on $\mathcal{M}_m$ s.t. 
$$h(a) = \int_{\mathcal{M}_m}K([a], \omega)d\nu(\omega) $$
where the Martin kernel and Martin boundary are obtained from the potential function  $\phi(x) = \log P((x)_0,(x)_1)$. 
\item
For every $a\in S$, and $A \subseteq \mathcal{M}_m$,
$${\Pr}_a[ \lim_{n\rightarrow\infty} T^n Z\in A] = \int_A K([a], \omega)d\nu_1(\omega) $$
where $\nu_1$ is the measure from  \ref{thm:prob_martin_boundary1}) with the harmonic function $h\equiv1$.
\end{enumerate}
\end{thm}
\begin{remk}
\label{remark:P_harmonic_func}
Notice that a $P$-harmonic function is not an eigenfunction of $L_\phi$, but rather an eigenfunction of  $L_{\phi^-}$, with $\phi^- (y) = \log P((y)_{-1}, (y)_0)$.
In particular, in the following proof, to simplify calculations we explicitly present  the correspondence between the $P$-harmonic functions and the conformal measures, rather than  applying Theorem \ref{thm:harm_conf_correspondence} or Theorem \ref{thm:harmonic_rep}.
\end{remk}
\begin{proof}
 Since
$$(L_\phi^n 1_{[a]})(x) = p^n((x)_0,a) $$
then
$$G(1_{[a]},x) = \sum_{n=0}^\infty p^n((x)_0,a)<\infty $$
and $\phi$ is indeed transient. \begin{enumerate}
\item 
We define a measure  
\begin{align}
\label{eq:func_to_measure}
&\mu([a]) = h(a) \\
&\mu([a_1,\dots, a_n]) =  e^{\phi(a_1,a_{2})}\cdots e^{\phi(a_{n-1},a_{n})}h(a_n).\nonumber
\end{align}
Since
\begin{align*} \sum_{b}\mu([a_1,\dots, a_n,b]) =& \sum_{b}e^{\phi(a_1,a_{2})}\cdots e^{\phi(a_{n-1},a_{n})}e^{\phi(a_n,b)}h(b)\\
=&e^{\phi(a_1,a_{2})}\cdots e^{\phi(a_{n-1},a_{n})}h(a_n)\\
=& \mu([a_1,\dots, a_n])
\end{align*}
$\mu$ can be extended to a measure via Carath\'eodory's extension theorem.
Moreover, since
\begin{align*}
(L_\phi^* \mu)([a])=& \mu(L_\phi 1_{[a]})\\
=&  \int e^{\phi(ax)}d\mu(x)\\
=& \sum_{b}e^{\phi(ab)}\mu([b])\\
=& \sum_{b}e^{\phi(ab)}h(b) = h(a) = \mu([a])
\end{align*}
and
\begin{align*} (L_\phi^* \mu)([a_1,\dots, a_n])=& \mu(L_\phi 1_{[a_1,\dots, a_n]})\\
=& \int \sum_{b}e^{\phi(bx)}1_{[a_1,\dots, a_n]}(bx)d\mu(x) \\
=& \int e^{\phi(a_{1}a_{2})}1_{[a_2,\dots, a_n]}(x)d\mu(x)\\
 =&e^{\phi(a_{1}a_{2})}\mu([a_2,\dots, a_n])=\mu([a_1,\dots, a_n])
\end{align*}
we have that $\mu\in \conf$. This establish a linear $1-1$ correspondence between the positive  $P$-harmonic functions and the conformal measures. 
 Theorem \ref{thm:minimal_kernel_rep} implies that there exists a unique measure $\nu$ s.t.
$$ \mu = \int _{\mathcal{M}_m}\mu_\omega d\nu(\omega)$$
and in particular
$$h(a) = \mu[a] = \int_{\mathcal{M}_m}K([a], \omega)d\nu(\omega).$$
\item
Since  $\sum_{b}P(a,b)=1$, for every $a\in S$, $1$ is a positive $P$-harmonic function. Let $\mu_{1}$ be the measure following Eq. (\ref{eq:func_to_measure}) with $h \equiv1$, and let $\nu_1$ the corresponding measure from Theorem \ref{thm:minimal_kernel_rep}. Fix $A \subseteq \mathcal{M}_m$, and let
$$\widetilde A = \{x\in X^{+} : T^n x \rightarrow A\}. $$
By definition,  
$$ \mu_{1}([a,a_{1},\dots, a_n]) = {\Pr}_{a}[Z_{1}=a_1,\dots Z_{n}=a_n].$$
Therefore, for any event $B \subseteq X^{+}$, 
$$\mu_{1}([a]\cap B) = {\Pr}_{a}[B]. $$
According to Theorem \ref{thm:conv_to_boundary}, for every $\omega\in \mathcal{M}_m$, $$\mu_{\omega}([a]\cap \widetilde A)=\begin{cases}\mu_{\omega}([a]) & \omega \in A \\
0 & o.w. \\
\end{cases}=\mu_{\omega}([a]) 1_{A}(\omega). $$
Thus, 
\begin{align*}
{\Pr}_{a}[\lim_{n\rightarrow\infty} T^{n }Z \in A]=&\mu_{1}([a]\cap \widetilde A)\\
=& \int_{\mathcal{M}_m}  \mu_\omega([a] \cap \widetilde A)d\nu_1(\omega)\\
=& \int_{\mathcal{M}_m}  \mu_\omega([a] )1_{A}(\omega)d\nu_1(\omega)\\
=& \int_{A}  K([a],\omega)d\nu_1(\omega). 
\end{align*}
\end{enumerate}
\end{proof}

\appendix 
\section{Appendix}
\begin{lem}
\label{lemma:G_continuous}
Let $(X^{+},T)$ be a transitive locally compact one-sided TMS and let $\phi$ be a $\lambda$-transient potential function with summable variations. Then, for every $f\in C_c(X)$ the Green's function $G(f, x|\lambda)$ and the Martin kernel $K(f,x|\lambda)$ are continuous in $X^+$.
\end{lem}
\begin{proof}
Assume w.l.o.g. that $\lambda=1$.

 Assume first that $f = 1_{[b_1,\dots, b_m]}$ for some admissible $b_1,\dots, b_m\in S$. Let $a\in S$ and let $x,y\in [a]$ with $d(x,y) < 2^{-N}$ namely $(x)_i = (y)_i$ for all $0\leq i\leq N$.
Given $a_0,\dots, a_n\in S$ admissible with $a_n=a$,
$$e^{\phi_n(a_0,\dots, a_{n-1}x)} =e^{\phi_n(a_0,\dots, a_{n-1}y)\pm \sum_{k=N}^{N+n}Var_k(\phi)} = e^{\phi_n(a_0,\dots, a_{n-1}y)\pm \sum_{k=N}^{\infty}Var_k(\phi)}. $$  
Moreover, when  $N> m$, 
$$1_{[b_1,\dots, b_m]}(a_0, \dots, a_{n-1}x)=1_{[b_1,\dots, b_m]}(a_0, \dots, a_{n-1}y) $$
and 
\begin{align*} &e^{\phi_n(a_0,\dots, a_{n-1}x)}1_{[b_1,\dots, b_m]}(a_0, \dots, a_{n-1}x)\\ &=e^{\pm \sum_{k=N}^{\infty}Var_k(\phi)}e^{\phi_n(a_0,\dots, a_{n-1}y)}1_{[b_1,\dots, b_m]}(a_0, \dots, a_{n-1}y). 
\end{align*}
Therefore, for all $N> m$, 
$$G(1_{[b_1,\dots, b_m]}, x) = e^{\pm \sum_{k=N}^{\infty}Var_k(\phi)}G(1_{[b_1,\dots, b_m]},y). $$
In particular, $G(1_{[b_1,\dots, b_m]}, x)$ is continuous.

Since $X^+$ is topologically transitive, the Green's function  $G(1_{[o]},x)$ is always non-zero and  the Martin kernel $K(1_{[w]},\cdot) = \frac{G(1_{[w]}, \cdot)}{K(1_{[o]},\cdot)}$ is a continuous function, for all admissible word $w\in S^*$.
 
Let  $f\in C_c^+(X^+)$ be an arbitrary non-zero. 
 Since $supp(f)$ is compact, there exist  $a_{i_1},\dots, a_{i_M}\in S $ s.t. $supp(f) \subseteq \cup_{j} [a_{i_j}]$. Moreover, since the collection $\{1_{[w]}\}_{w\in S^*}$ spans linearly a dense subset of $C_c(X^{+})$ w.r.t. the sup-norm $||\cdot||_\infty$, for every $\epsilon>0$ we can find $w_1,\dots, w_N\in S^*$ and $\alpha_1,\dots, \alpha_N\in \mathbb{R}^+$ s.t. $\cup_{i=1}^{N}[w_i] \subseteq \cup_{j=1}^M [a_{i_j}]$ and 
$$||f - \sum_{i=1}^N \alpha_i1_{[w_i]}||_\infty < \epsilon. $$
Then, for some constant $C=C(f)>0$,
$$||   K(f,\cdot) -   K(\sum_{i=1}^N \alpha_i1_{[w_i]},\cdot)||_\infty \leq \epsilon||  K(1_{\cup_{j} [a_{i_j}]},\cdot)||_{\infty} \leq \epsilon \sum_{j=1}^M C_{a_{i_j}}\leq \epsilon \cdot C $$
where $C_{a}$ is the constant from Lemma \ref{lemma:kernel_boundness}. Fix $x\in X^+$. Since the function $K(\sum_{i=1}^N \alpha_i1_{[w_i]},\cdot)$ is continuous, there exists $\delta>0$ s.t. for all $y\in X^+$ with $d(x,y)<\delta$,
$$|K(\sum_{i=1}^N \alpha_i1_{[w_i]},x) - K(\sum_{i=1}^N \alpha_i1_{[w_i]},y)| < \epsilon. $$
Then, for all such $y$,  
\begin{align*}|K(f,x) - K(f,y)| \leq  & |   K(f,x) -   K(\sum_{i=1}^N \alpha_i1_{[w_i]},x)|  \\
+& |   K(f,y) -   K(\sum_{i=1}^N \alpha_i1_{[w_i]},y)|  \\
+& |K(\sum_{i=1}^N \alpha_i1_{[w_i]},x) - K(\sum_{i=1}^N \alpha_i1_{[w_i]},y)| \\
\leq& (2C+1)\epsilon.   
\end{align*}
This implies that  $K(f,\cdot)$ is continuous, for all $f\in C_c(X^+)$. 

Lastly, since the Green's function  $G(f,\cdot )$ is a multiplication of two continuous functions,  $G(f,\cdot ) = K(f,\cdot)G(1_{[o]},\cdot)$, it is a continuous function as well.

\end{proof}
\begin{lem}   
\label{lemma:positive_on_cylinders}
Let $(X^{+},T)$ be a transitive locally compact one-sided TMS and let $\phi$ be a continuous potential function. Then, for every non-zero $\mu\in \conf(\lambda)$ and $a\in S$, $\mu([a])>0$.  \end{lem}
\begin{proof}
Let $a\in S$. Since $\mu\neq 0$, there is a state $b\in S$ s.t. $\mu([b]) >0$. Let $a_0,\dots, a_n\in S$ be an admissible word from $a_0=a$ to $a_n=b$. Then, 
\begin{align*}
\mu([a]) = &\ \lambda^{-n}\mu(L_\phi^n 1_{[a]})\\
\geq & \lambda^{-n}\int e^{\phi_n(aa_1\dots a_{n-1}x)}1_{T[a_{n-1}]}(x)d\mu(x)\\
\geq &  \lambda^{-n}\int e^{\phi_n(aa_1\dots a_{n-1}x)}1_{[b]}(x)d\mu(x)\\
\geq &\ \lambda^{-n}\min_{x\in [b]}\left(e^{\phi_n(aa_1\dots a_{n-1}x)}\right)\mu([b])>0.
\end{align*} 
\end{proof}

\begin{proof}[\bf{Proof of Proposition \ref{claim:compactification_prop}}]
Assume w.l.o.g. that $\lambda =1$.
\begin{enumerate}[label={(\arabic*)}]
\item 
Assume that $x_n \xrightarrow{d}x\in X^{+}$. Let $\epsilon>0$ and $m_{}$ be large enough s.t.
$$ \sum_{i\geq m} \frac{|  K(1_{[w_i]},x)-  K(1_{[w_i]},y)|+|1_{[w_i]}(x)-1_{[w_i]}(y)|}{2^i(C_{1_{[w_i]}}+1)}<\epsilon/2, \quad \forall x,y,\in X^{+}.$$
Since $  G(f,x)$ and $  K(f,x)$ are continuous functions of $x$ for fixed every $f\in C_c(X^{+})$ (see Lemma \ref{lemma:G_continuous}), we can find  $n$ large enough s.t.
$$ \sum_{i< m} \frac{|  K(1_{[w_i]},x_{n})-  K(1_{[w_i]},x)|+|1_{[w_i]}(x_{n})-1_{[w_i]}(x)|}{2^i(C_{1_{[w_i]}}+1)}<\epsilon/2.$$ Then $\rho(x_n, x) < \epsilon$  and $\{x_n\}$  converges to $x$ w.r.t. $\rho$.

Next, assume that $x_n\xrightarrow{\rho}x\in X^{+}$. Since for every finite word $w\in S^*$, there exists $C>0$ s.t. $$|1_{[w]}(x)-1_{[w]}(x_n)|\leq C \rho(x,x_n) \xrightarrow[n\rightarrow\infty]{} 0$$
and $|1_{[w]}(x)-1_{[w]}(x_n)| \in \{0,1\}$, we have that $1_{[w]}(x)=1_{[w]}(x_n)$ for $n$ large enough. Therefore, for every $m$, we can find $n$ large enough s.t. $(x_n)_i = (x)_i$ for every $1\leq i\leq m$, meaning $d(x_n,x) \leq 2^{-m}$.
\item 
If there exists a compact set $A\subseteq X^{+}$ s.t. $x_n\in A$ infinitely many times, then we can find a sub-sequence $x_{n_k}\in A$ with $x_{n_k}\xrightarrow[k\rightarrow\infty]{d}x\in A$. But then $x_{n_k}\xrightarrow[k\rightarrow\infty]{\rho}x\in A \subseteq X^{+}\  $ which contradicts the convergence of $x_n$ to a boundary point.

\item         
Let $\{x_n\}$ be an arbitrary sequence in $\widehat{X^+}$. By definition, $X^{+}$ is dense in $\widehat{X^+} $. If $\{x_n \} \cap \mathcal{  M} \neq \varnothing$, we replace every point $x_n\in \mathcal{  M}$ with some $x_n'\in X^{+}$ s.t. $\rho(x_n, x_n') < 1/n$. We will use diagonalization  argument to show that $\{x_n'\}$ has a Cauchy sub-sequence. Since $\rho(x_n, x_n')\leq 1/n$, the original sequence will have a Cauchy sub-sequence as well. Thus we can assume w.l.o.g. that $\{x_n\} \subseteq X^{+}$.

Let $f_i(x) = \frac{1}{2^i(C_{1_{[w_i]}}+1)}  K(1_{[w_i]},x|\lambda)$ and $g_i(x) =\frac{1}{2^i(C_{1_{[w_i]}}+1)}1_{[w_i]}(x) $. Then, 
$$\rho(x,y) = \sum_{i}\left(|f_i(x) - f_i(y)| + |g_i(x) - g_i(y)|\right). $$
Since $g_1$ and $f_1$ are bounded, we can find a sub-sequence $x_{n_k^1}$ s.t. $g_1(x_{n_{k}^1})$ and $f_1(x_{n_{k}^1})$ converge. Similarly, for every $m$ we can find $x_{n_k^m}$ a sub-sequence of $x_{n_k^{m-1}}$ s.t. $g_m(x_{n_k^m})$ and $f_m(x_{n_k^m})$ converge. Set $x_{n_k}=x_{n_k^{k}}$. We show that $x_{n_k^k}$ is a Cauchy sequence. Let $\epsilon>0$, and let $N$ be large enough s.t.
$$\sum_{i\geq N}\left (|f_i(x) - f_i(y)| + |g_i(x) - g_i(y)|\right) \leq \epsilon, \quad \forall x,y\in X^+.$$
Let $K$ be large enough s.t. for every $k_{1},k_2\geq K$ and every $i<N$
$$|f_i(x_{n_{k_1}}) - f_i(x_{n_{k_2}})| + |g_i(x_{n_{k_1}}) - g_i(x_{n_{k_1}})|\leq \epsilon/N.   $$
Then, for all $k_1,k_2\geq K$, 
$$\rho(x_{n_{k_1}},x_{n_{k_2}})\leq \sum_{i< N}\left (|f_i(x_{n_{k_1}}) - f_i(x_{n_{k_2}})| + |g_i(x_{n_{k_1}}) - g_i(x_{n_{k_2}})|\right) + \epsilon \leq 2\epsilon.  $$
Hence $\{x_{n_k}\}$ is a  $\rho$-Cauchy sequence and $\widehat {X^+}$ is compact. 

Next we show that  $\mathcal{M}$ is closed. Let $\omega_n \in \mathcal{M}_m$ be a sequence of boundary points which $\rho$-converges to a point $x\in \widehat{X^+}$. Assume by contradiction that $x\in X^+$. For every $n$, let $x_n \in X^+$ s.t. $\rho(x_n, \omega_n)\leq 1/n$. Then, $x_n \xrightarrow[]{\rho}x$ and by part  \ref{claim:compactification_prop_2} of the proposition  $x_n\xrightarrow[]{d}x$. By part \ref{claim:compactification_prop_escapes_compact} of the proposition, $x_n$ can be chosen so that $x_n \not\in[(x)_0]$ for all $n$, which contradicts the $d$-convergence of $x_n$ to to $x$.       

\item 
Assume that $A$ is a $d$-open subset of $X^{+}$. We show that $A^c =  \widehat{X^+}  \setminus A$ is $\rho$-closed. Let $x_n \in A^{c}$, $x_n\xrightarrow[]{\rho}x$. If $x\in \mathcal{  M}$, then clearly $x\not\in A$. If $x\in X^{+}$, since $X^+$ is $\rho$-open, for $n$ large enough $x_n \in X^+$ and $x_n\xrightarrow[]{d}x$ as well. Since $x_n \not\in A$ for all $n$, we must have that $x\not\in A$ and thus $A^{c}$ is indeed $\rho$-closed.

Next, assume that $A$ is $\rho$-open. We show that $B=A^{c}\cap X^{+}$ is $d$-closed. Let $x_n \in B$, $x_n\xrightarrow[]{d}x\in X^{+}$. Then, $x_n\xrightarrow[]{\rho}x$ as well. Since $x_n \not\in A $ and $A$ is a $\rho$-open set, $x\not \in A$, whence $x\in B$. 

\item By the construction of $\widehat{X^+} $ and $\rho$, for every $w\in S^*$,  $  K(1_{[w]},\cdot):X^+\rightarrow \mathbb{R}$ extends uniquely to a $\rho$-continuous function on $ \widehat{X^+} $. Let $f\in C_c(X^{+})$, and let $x_n\in X^{+}, \omega\in\mathcal{M}$ with $x_n\rightarrow x$. We show that the limit $  \lim_{n\rightarrow\infty}K(f,x_n)$ exists and does not depend on $x_n$ but only on $\omega$.  Since $supp(f)$ is compact, there exist  $a_{i_1},\dots, a_{i_M}\in S $ s.t. $supp(f) \subseteq \cup_{j} [a_{i_j}]$. Moreover, since the collection $\{1_{[w]}\}_{w\in S^*}$ spans linearly a dense subset of $C_c(X^{+})$ w.r.t. the sup-norm $||\cdot||_\infty$, for every $\epsilon>0$ we can find $w_1,\dots, w_N$ and $\alpha_1,\dots, \alpha_N\in \mathbb{R}^+$ s.t. $\cup_{i=1}^{N}[w_i] \subseteq \cup_{j=1}^M [a_{i_j}]$ and 
$$||f - \sum_{i=1}^N \alpha_i1_{[w_i]}||_\infty < \epsilon. $$
Then, for some constant $C=C(f)>0$,
$$||   K(f,\cdot) -   K(\sum_{i=1}^N \alpha_i1_{[w_i]},\cdot)||_\infty \leq \epsilon||  K(1_{\cup_{j} [a_{i_j}]},\cdot)||_{\infty} \leq \epsilon \sum_{j=1}^M C_{a_{i_j}}\leq \epsilon \cdot C $$
where $C_{a}$ is the constant from Lemma \ref{lemma:kernel_boundness}. 
From this we deduce that $  K(f,x_n)$ is a Cauchy sequence of real numbers  and that the limit $\lim_{n\rightarrow\infty}  K(f,x_n)$ exists. Let $K(f,\omega) = \lim_{n\rightarrow\infty}  K(f,x_n)$. Since the limit $K(1_{[w_i]},\omega) =\lim_{n\rightarrow\infty}K(1_{[w_i]},x_{n})$  does not depend on the choice of $x_n$, so is $K(f,\omega)$. 
\item
Since $  K(\cdot,x)$ defines a positive linear functional on $C_c^{+}(X^{+})$, we can apply the Riesz representation theorem to obtain a Radon measure $\mu_x$ s.t. $\mu_x(f) =   K(f,x)$ for all $f\in C_c(X^{+})$. Moreover, for every $x\in X^{+}$ and $f\in C_c(X^{+})$,
$$( L_\phi^* \mu_x )(f) =  K(L_\phi f, x) =   K(f,x) -\frac{f(x)}{  G(1_{[o]},x)}\leq   K(f,x)= \mu_x(f) $$
meaning $\mu_x$ is excessive. Observe that for all $a\in S$,  $supp(L_\phi 1_{[a]})=\cup_{b:\mathbb{A}_{a,b}=1}[b]$. In particular, since $\sum_{b\in S}\mathbb{A}_{a,b}<\infty$ the support of $L_\phi f$ is compact as well, so   $L_\phi f  \in C_c(X^{+})$  and  $  K(L_\phi f, \cdot)$ is $\rho$-continuous. Then, given a boundary point $\omega\in \mathcal{  M}$, if $x_n\xrightarrow[]{\rho}\omega$ then $x_n\xrightarrow[]{d}\infty$, whence $f(x_n)=0$ eventually. In particular, 
\begin{align*} (L_\phi^* \mu_\omega) (f) =&  K(L_\phi f, \omega) \\
=&\lim_{x\rightarrow\omega}   K(L_\phi f, x)\\
=&\lim_{x\rightarrow\omega}(  K(f,x) -\frac{f(x)}{  G(1_{[o]},x)})\\
=&\lim_{x\rightarrow\omega}  K(f,x)=  K(f,\omega) \end{align*}
 and $\mu_\omega$ is conformal.
We remark that $\mu_\omega([o])=  K(1_{[o]},\omega)=1$, so $\mu_\omega \neq 0$.
\end{enumerate}
\end{proof}

\begin{prop}
\label{prop:minimal_borel}
Let $(X^{+},T)$ be a transitive locally compact one-sided TMS and let $\phi$ be a $\lambda$-transient potential function with summable variations. Then, the minimal boundary $\mathcal{M}_m(\lambda)$ is a Borel set w.r.t. the weak$^*$-topology.
\end{prop}
\begin{proof}
The following arguments are inspired by the proof of Proposition 1.3 in \cite{phelps_2001}. 

Let $\conf_o(\lambda) = \{\mu\in \conf(\lambda) : \mu([o])=1\}$. First we show that $\conf_o(\lambda)$ is a compact set in the weak$^*$ topology. Observe that  Theorem \ref{thm:boundary_rep_theorem} and Lemma \ref{lemma:kernel_uniform_boundness} together imply that for all $\mu \in \conf_o(\lambda)$ and all $f\in C_c^+(X^+)$, \begin{equation}
\label{eq:uniform_measure_bound}
\mu(f)\leq C_f
\end{equation}
where $C_f$ is the bound from Lemma \ref{lemma:kernel_uniform_boundness}. In particular, this uniform bound and the fact that $1_{[o]}$  is a continuous functions in $X^+$ imply that every sequence in $\conf_o(\lambda)$ has a subsequence that converges  w.r.t. the weak$^*$ topology to a measure $\mu$ with $\mu([o])=1$.  Since $L_\phi f\in C_c(X^+)$ for all $f\in C_c(X^+)$ (see proof of part \ref{claim:compactification_prop_6} in Proposition \ref{claim:compactification_prop}), the limiting measure $\mu$ is a conformal measure as well and thus $\conf_o(\lambda)$ is  a compact set.

 We define a metric on $\conf_o(\lambda)$: for $\mu_1,\mu_2\in \conf_o(\lambda)$, let 
$$ \delta(\mu_1, \mu_2) = \sum_{i=1}^\infty \frac{|\mu_1([w_i]) - \mu_2([w_i])|}{C_{1_{[w_i]}}2^i}$$
where $\{w_i\}$ is an enumeration of all admissible words over $S$ and $C_{1_{[w]}}$ is the bound from Lemma \ref{lemma:kernel_uniform_boundness}. The uniform bound in Eq. (\ref{eq:uniform_measure_bound}) implies that $\delta$ is always finite and thus it is indeed a metric.    

Let $$F_N = \left\{\mu \in \conf_o(\lambda) : \exists \mu_1,\mu_2\in \conf_o(\lambda)  \text{ s.t. } \mu = \frac{\mu_1 + \mu_2}{2} \text{ and }\delta(\mu_1, \mu_2) \geq \frac{1}{N}\right\}.$$
We show for all $N\in \mathbb{N}$, the set  $F_N$ is closed in the weak$^*$ topology. Let $\mu_n\in F_{N}$ s.t. $\mu_n \rightarrow \mu$ and let $\mu_1^n,\mu_2^n\in\conf_o(\lambda)$ s.t. $\mu_n = \frac{\mu_1^n + \mu_2^n}{2}$. Since $\conf_o(\lambda)$ is compact, there exist  converging sub-sequences $\mu_1^{n_k}, \mu_2^{n_k}$ to $\mu_1,\mu_2\in \conf_o(\lambda)$ respectively and $\mu = \frac{\mu_1 + \mu_2}{2}\in \conf_o(\lambda)$. To show that $\delta(\mu_1,\mu_2)\geq \frac{1}{N}$, observe that for all $\mu_1', \mu_2'\in \conf_o(\lambda)$,  
$$\sum_{i> m} \frac{|\mu'_1([w_i]) - \mu'_2([w_i])|}{C_{1_{[w_i]}}2^i} \leq \sum_{i>m}2^{-i+1} = 2^{-m+1}.$$
Then, since $1_{[w_i]}$ is continuous for all $i$, 
\begin{align*}
\sum_{i= 1}^m \frac{|\mu_1([w_i]) - \mu_2([w_i])|}{C_{1_{[w_i]}}2^i} 
=  &\lim_{k\rightarrow\infty} \sum_{i=1}^m \frac{|\mu_1^{n_k}([w_i]) - \mu_2^{n_k}([w_i])|}{C_{1_{[w_i]}}2^i}\\
\geq & \liminf_{k\rightarrow\infty} \sum_{i=1}^\infty \frac{|\mu_1^{n_k}([w_i]) - \mu_2^{n_k}([w_i])|}{C_{1_{[w_i]}}2^i} - 2^{-m+1} \\
\geq& \frac{1}{N}-2^{-m+1}.
\end{align*}
Passing to the limit as $m\rightarrow\infty$, we obtain that $\delta(\mu_1, \mu_2) \geq \frac{1}{N}$ and $F_N$ is indeed a closed set.

By definition, $\omega_n \xrightarrow[]{\rho}\omega$ iff $\mu_{\omega_n}\xrightarrow[]{\text{weak}^*}\mu_{\omega}$. Therefore the set $\{\omega \in \mathcal{M}(\lambda) : \mu_\omega \in F_N\}$ is $\rho$-closed and in particular a Borel set.
Thus to conclude the proof it suffices to show that 
$$\mathcal{M}_m(\lambda) = \mathcal{M}(\lambda) \setminus \bigcup_{N=1}^\infty \{\omega \in \mathcal{M}(\lambda) : \mu_\omega \in F_N\}. $$
Clearly $\{\omega \in \mathcal{M}(\lambda) : \mu_\omega \in F_N\} \cap \mathcal{M}_m(\lambda) = \varnothing$, so we have $\subseteq$. We prove $\supseteq$. 

Let $\omega \in \mathcal{M}(\lambda) \setminus \mathcal{M}_m(\lambda)$. We show that $\mu_\omega \in F_N$ for some $N>0$. By the construction of the Martin kernels, $\mu_\omega([o])=1$ and thus $\mu_\omega \in \conf_o(\lambda)$. Since $\mu_\omega$ is not extremal in $\conf(\lambda)$, we can find two non-proportional measures $\mu_1, \mu_2\in \conf(\lambda)$ s.t. $\mu = \mu_1 + \mu_2$. By Lemma \ref{lemma:positive_on_cylinders}, $\mu_i([o])>0$. Let $\mu'_i = \frac{1}{\mu_i([o])}\mu_i$. Notice that since $\mu_1$ and $\mu_2$ are not proportional, $\mu_1'$ and $\mu_2'$ are not proportional as well.  Now, 
$\mu = \mu_1([o])\mu_1' + \mu_2([o]) \mu_2'$.  Let $t = \mu_1([o])$. Then, we can write 
$$\mu = \frac{t \mu_1' + (1-t) \mu}{2} + \frac{(1-t) \mu_2' + t \mu}{2}. $$ 
Let $\mu_1 '' = t \mu_1' + (1-t) \mu$ and let $\mu''_2 =(1-t) \mu_2' + t \mu$. Clearly $\mu_1'',\mu_2'' \in \conf_o(\lambda)$. They are different, otherwise $\mu_1$ and $\mu_2$ would be proportional to $\mu$. Therefore, there exists $N>0$ s.t. $\delta(\mu_1'', \mu_2'') \geq \frac{1}{N}$ and $\mu_\omega \in F_N$. 
\end{proof}
\begin{proof}[\bf{Proof of Proposition \ref{prop:dissipative_on_compacts}}]
Assume w.l.o.g. that $\lambda=1$ and assume first that $F=[a]$ for some $a\in S$. Since $G(1_{[a]},x)$ is continuous,
it is bounded on compacts and thus, for every $b\in S$,
$$ \mu\left(1_{[b]}\sum_{n\geq0}L_\phi^n 1_{[a]}\right)<\infty.$$
Since $\mu$ is conformal, $\mu(f L_\phi g) = \mu((f\circ T) \cdot g)$, $\forall f,g\in C_c(X^+)$. Hence, 
$$ \mu\left(1_{[b]}\sum_{n\geq0}L_\phi^n 1_{[a]}\right)=\mu\left(1_{[a]}\sum_{n\geq0} 1_{[b]}\circ T\right)<\infty.$$
In particular, $1_{[a]}\sum_{n\geq0} 1_{[b]}\circ T<\infty$ almost-surely, whence for $\mu$-a.e. $x\in [a]$,  $T^nx \in [b]$ finitely many times.

Assume now that $F$ is arbitrary compact. Then, there exist $a_1,\dots, a_N\in S$ s.t. $F \subseteq \cup_{i=1}^N [a_i]$. Observe that in order to return infinitely many times to $F$ we must return infinitely many times to one of the cylinders $[a_i]$. Therefore,  $F_{\infty} \subseteq\cup_{i=1}^N [a_i]_{\infty}$ and $\mu(F_{\infty}) \leq \sum_{i=1}^N \mu([a_i]_{\infty})=0$.
\end{proof}

\begin{proof}[\bf{Proof of Proposition \ref{claim:left_right_transience}}]
Assume w.l.o.g. that  $\lambda=1$. By symmetry, we only show that if $\phi^+$ is transient then $\phi^-$ is transient as well. In particular, we show that there exists a non-zero function $f\in C_c^{+}(X^-)$ and a point $y\in X^-$ s.t. 
$$ \sum_{n=0}^\infty L_{\phi^-}^n (f)(x)<\infty.$$
 
Fix $a\in S$,  $x\in X^+\cap T[a], y\in X^-\cap T^{-1}[a]$. For every admissible word $(a_1,\dots,a_n)$ with $a_1=a_{n}=a$, we write  $$z=(\dots, (y)_{-1}\dot {,(y)}_{0},a_{1},\dots, a_n,(x)_{0},(x)_1,\dots).$$  
 Let $C_{1} = \max_{z'\in [a]}|\psi(z')|$ and $C_{2} = \max_{z'\in T [a]}|\psi(z')|$. Then,\begin{align*} |\psi(Tz) |=|& \psi(y\dot a_{1},\dots, a_nx)  |\leq C_{1}  
\end{align*}
$$ |\psi(T^{n+1}z)| = |\psi(ya_{1},\dots, a_n\dot x)|\leq C_2 .$$
From the cohomology property in Eq. (\ref{eq:cohomo_eq}), 
$$\phi^+(a_{i},\dots, a_nx) - \phi^-(y a_{1},\dots, a_i) = \psi(T^{i}z) -\psi(T^{i+1}z).$$
In particular, 
\begin{align*}\phi^+_{n}(a_1,\dots, a_nx)  = &\phi_{n}^-(ya_1,\dots, a_n)+ \psi(Tz)- \psi(T^{n+1} z)\\
\geq &\phi_{n}^-(ya_1,\dots, a_n)- C_1 -C_2.
\end{align*}
Then, for every  $n\geq 2$ 
\begin{align*}
 L_{\phi^+}^n (1_{[a]})(x) = &  \sum_{\substack{a_1,\dots, a_n\\a_1=a}}e^{\phi^+_{n}(a_1,\dots,a_nx)}\\
\geq &
 \sum_{\substack{a_1,\dots, a_n\\a_1,a_n=a}}e^{\phi^+_{n}(a_1,\dots,a_nx)}\\
\geq &e^{-C_{1}-C_2} \sum_{\substack{a_1,\dots, a_n\\a_1,a_n=a}}e^{\phi_{n}^-(ya_1,\dots, a_n)  } \\
=& e^{-C_{1}-C_2+\phi^-(ya)} \sum_{\substack{a_2,\dots, a_n\\a_n=a}}e^{\phi_{n-1}^-(yaa_2,\dots, a_n)}\\
=& e^{-C_{1}-C_2+\phi^-(ya)}L_{\phi^-}^{n-1}(1_{[a]})(ya).
\end{align*}
This implies that
\begin{align*}
G(1_{[a]},x) \geq& e^{-C_{1}-C_2+\phi^-(ya)} \sum_{n\geq1}L_{\phi^-}^{n}(1_{[a]})(ya)\\
 =& e^{-C_{1}-C_2+\phi^-(ya)} \left(\sum_{n=0}^\infty G(1_{[a]},ya) - 1\right)
\end{align*}
and  $\phi^-$ is indeed transient.
\end{proof}

\begin{proof}[\bf{Proof of Proposition \ref{Prop-DLRisConformal}}]

Recall that for a non-singular $\nu$, there exists a Borel function $\phi'$ s.t. $\frac{d\nu}{d\nu\circ T} = \exp \phi'$, where $(\nu\circ T )(A) := \sum_{a\in S}\nu(T(A\cap [a]))$ and in particular $L_{\phi'}^*\nu = \nu$.  By Proposition \ref{Prop-ConformalIsDLR},  $\nu$ is a DLR state of $\phi'$. 
 Given two admissible words $\underline a = (a_1,\dots, a_n)$ and $\underline b=(b_1,\dots, b_n)$ of length $n$ with $a_n=b_n$, let 
$$ \vartheta_{\underline a,\underline b}:[\underline a]\rightarrow[\underline b], \; \vartheta_{\underline a,\underline b}(\underline a x_{n}^\infty ) = \underline b x_{n}^\infty$$
where $x_n^\infty=((x)_n, (x)_{n+1},\dots)=T^{n }x$.
The following claim is elementary, see \cite{petersen_1997,sarig_2015}.
\begin{claim}
\label{claim:DLR_property}
A non-singular  probability measure $\mu$ is a  DLR for measure for $\phi$ iff for every two admissible words $\underline a = (a_1,\dots, a_n)$ and $\underline b=(b_1,\dots, b_n)$ of length $n$ with $a_n=b_n$, $\text{for a.e. }x\in T^{-n}[a_n] $,
\begin{equation}
\label{eq:dlr_condition_2}
\frac{d\mu\circ\vartheta_{\underline a,\underline b} }{d\mu}=\exp \left(\sum_{k=1}^n\left((\phi(a_{k},\dots ,a_nx_n^\infty) - \phi(b_{k},\dots ,b_nx_n^\infty)\right)\right).
\end{equation}
\end{claim}
Since $\mu$ is a DLR state both for $\phi$ and $\phi'$, Eq. (\ref{eq:dlr_condition_2}) implies that$$ \sum_{k=0}^{n-1}((\phi\circ T^k)( \underline a x_n^\infty) -(\phi\circ T^k)( \underline b x_n^\infty))  =  \sum_{k=0}^{n-1}((\phi'\circ T^k)( \underline a x_n^\infty) -(\phi'\circ T^k)( \underline b x_n^\infty))  ,\quad \nu-a.s. $$   
for every $\underline a, \underline b$ of length $n$ with $a_n=b_n$ and $\mu[\underline a]>0$.
 
Let $h = \phi - \phi'$.  We show by induction that for any such admissible $\underline a, \underline b$ and for $\nu$-a.e. $x\in X^+$, $h(\underline a x_{n}^\infty) = h(\underline b x_{n}^\infty)$. This will imply that $h$ is measurable w.r.t. the $\sigma$-algebra $\cap_{i=0}^{\infty}T^{-i}\mathscr{B}$. Clearly the statement is true for $n=0$. For $n>0$, let $\underline a, \underline b$ be two admissible words of length $n$ which ends with the same symbol. By the induction assumption, for every $1<k\leq n-1$, 
 $$ (h\circ T^k)(\underline ax_n^\infty) = (h\circ T^k)(\underline b x_n^\infty), \quad \nu-a.s. $$
 Therefore, for $\nu$-a.e. $x\in T^{-n}[a_n]$, 
 \begin{align*}&h(\underline a x_n^\infty) -  h(\underline b x_n^\infty) = \sum_{k=0}^{n-1}\left((h\circ T^k)(\underline a x_n^\infty) -(  h\circ T^k)(\underline b x_n^\infty)\right) \\&=\sum_{k=0}^{n-1}\left((\phi\circ T^k)(\underline ax_n^\infty ) -( \phi\circ T^k)(\underline bx_n^\infty)-(\phi'\circ T^k)(\underline ax_n^\infty ) + (\phi'\circ T^k)(\underline bx_n^\infty)\right)\\
&=0. 
\end{align*}
\end{proof}

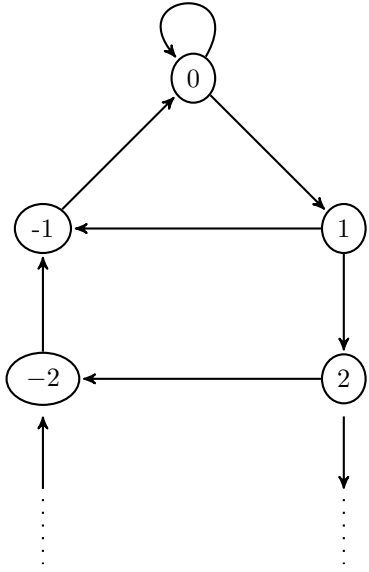
\begin{figure}[t!]
\label{figure:DLR_aperiodic_singular}
\begin{center}
\begin{tikzpicture}[-,>=stealth',shorten >=1pt,auto,
                    thick,main node/.style={circle,draw,font=\sffamily\Large\bfseries}]

  \node[ellipse,draw] at (0,0) (0) {$0$};
  \node[ellipse,draw] at (-2,-2) (n1) {-1};
  \node[ellipse,draw] at (-2,-4) (n2) {$-2$};
  \node[ellipse,draw] at (2,-2) (p1) {1};
  \node[ellipse,draw] at (2,-4) (p2) {$2$};

  \path[every node/.style={font=\sffamily\small}]
    (0) edge [thick, in=130,out=60,loop]  (0)
    (0) edge [->,thick]  (p1)    
    (n2) edge [->,thick]  (n1) 
    (n1) edge [->,thick]  (0)  
    (p1) edge [->,thick] (p2) 
    (p1) edge [->,thick] (n1) 
    (p2) edge [->,thick] (n2)

;

\draw [->, thick] (2,-4.5) -- (2,-5.5);
\draw [loosely dotted, thick] (2,-5.6) -- (2,-6.5);
\draw [<-, thick] (-2,-4.5) -- (-2,-5.5);
\draw [loosely dotted, thick] (-2,-5.6) -- (-2,-6.5);
 
\end{tikzpicture}
\end{center}
\caption{The state graph in Example 2.}
\end{figure}

\noindent\textbf{Example 2} (Non-conformal DLR measure).
\label{example:DLR_aperiodic_singular}
Consider $S=\mathbb{Z}$ and the transition matrix  
$$A_{a,b}=\begin{cases}
1 &    b=a+1 \text{ or } 
b=-a\leq 0  
\\
0 & otherwise \\
\end{cases}.$$
See Figure 2. It is easy to verify that the graph is locally finite, irreducible and aperiodic. Equip $X^+$ with any continuous potential function with summable variations $\phi:X^+\rightarrow\mathbb{R}$.   Let $x_0 = (0,1,2,3,4,5,\dots)$ and $\mu=\delta_{x_0}$. Clearly $\mu$ is a Radon measure. Since $\mu$ is supported on a single point and $\{y\in X^+:T^ny=T^nx_0\}=\{x_0\}$, Eq. (\ref{eq:dlr_condition}) holds trivially and  $\mu$ is a DLR\ measure. On the other hand, since $\mu([0])=1$ and $\mu(T^{-1}[0]) = \mu([0,0]\cup [-1,0])=0$, $\mu$ is non-conformal.

\section*{Acknowledgments}
I would   like to express my  gratitude  to my supervisor Prof. Omri Sarig for  the professional and moral support. I would like also  to thank the referees for the useful comments. This work is part of the author's Ph.D. dissertation at the Weizmann Institute of Science and  was also partially supported by Israel Science Foundation grants 199/14 and 1149/18.

%
\bibliography{bibfile}
\bibliographystyle{plain}

\end{document}